\documentclass[10pt,a4paper,reqno]{amsart}

\usepackage{amsmath}
\usepackage{amsfonts}
\usepackage{amssymb}
\usepackage{amsthm}
\usepackage[shortlabels]{enumitem}
\usepackage{graphicx}
\usepackage{color}
\usepackage{dsfont}
\usepackage[latin1]{inputenc}
\usepackage{MnSymbol}
\usepackage{enumitem}
\usepackage{extpfeil}
\usepackage{mathtools}
\usepackage{url}
\usepackage[margin=1.25in]{geometry}
\usepackage{fancyhdr}
\usepackage[all,cmtip]{xy}
\usepackage{hyperref}
\usepackage[capitalise]{cleveref}
\usepackage{tikz}
\usepackage{thm-restate}

\numberwithin{equation}{section}

\newcommand{\RR}{{\mathds R}}
\newcommand{\NN}{{\mathds N}}
\newcommand{\ZZ}{{\mathds Z}}

\newcommand{\Area}{{\rm{Area}}}

\def\ker{{\rm{ker}}}

\def\a{\mathfrak{a}}
\def\b{\mathfrak{b}}
\def\c{\mathfrak{c}}
\def\d{\mathfrak{d}}

\theoremstyle{plain}
\newtheorem{theorem}{Theorem}[section]

\newtheorem{corollary}[theorem]{Corollary}

\newtheorem{lemma}[theorem]{Lemma}
\newtheorem{proposition}[theorem]{Proposition}
\newtheorem{question}{Question}
\newtheorem*{question*}{Question}

\theoremstyle{definition}
\newtheorem{remark}[theorem]{Remark}
\newtheorem*{acknowledgements*}{Acknowledgements}
\newtheorem{example}[theorem]{Example}
\newtheorem{definition}[theorem]{Definition}
\newtheorem*{notation*}{Notation}
\newtheorem*{convention*}{Convention}

\newcommand*\circled[1]{\tikz[baseline=(char.base)]{
		\node[shape=circle,draw,inner sep=1pt] (char) {\tiny #1};}}

\title{Dehn functions of coabelian subgroups of direct products of groups}
\author{Robert Kropholler}
\address{Mathematics Institute, University of Warwick, Coventry CV4 7AL, United Kingdom}
\email{robertkropholler@gmail.com}
\author{Claudio Llosa Isenrich}
\address{Faculty of Mathematics, KIT, Englerstr. 2, 76131 Karlsruhe, Germany}
\email{claudio.llosa@kit.edu}

\thanks{The first author was funded by the Deutsche Forschungsgemeinschaft (DFG, German Research Foundation) under Germany's Excellence Strategy EXC 2044-390685587, Mathematics M\"unster: Dynamics-Geometry-Structure. The second author was partially funded by the Lise Meitner Fellowship M2811-N of the Austrian Science Fund (FWF)}

\begin{document}

\begin{abstract}
We develop new methods for computing the precise Dehn functions of coabelian subgroups of direct products of groups, that is, subgroups which arise as kernels of homomorphisms from the direct product onto a free abelian group. These improve and generalise previous results by Carter and Forester on Dehn functions of level sets in products of simply connected cube complexes, by Bridson on Dehn functions of cocyclic groups and by Dison on Dehn functions of coabelian groups. We then provide several applications of our methods to subgroups of direct products of free groups, to groups with interesting geometric finiteness properties and to subgroups of direct products of right-angled Artin groups.
\end{abstract}

\maketitle

\section{Introduction}
The properties of groups arising as kernels of maps onto $\ZZ$ have attracted a lot of attention in various areas of group theory. In particular, this concerns their finiteness properties (e.g. \cite{Sta-63, Bie-76, BesBra-97}). A quantitative approach to understanding the finiteness properties of groups is to study how difficult it is to detect if loops (or, more generally, spheres) in a $K(G,1)$ of a group $G$ are null-homotopic. For loops this is measured by the Dehn function $\delta_G(n)$ of the group $G$, which is defined as the maximal area that a minimal filling disc of a loop of length at most $n$ can have. The Dehn function constitutes an important asymptotic invariant with connections to fundamental problems in various areas, including in geometry in the form of isoperimetric functions and in combinatorial group theory in the form of the solvability of the word problem. 

The connection between finiteness properties and Dehn functions makes it particularly tempting to study the Dehn functions of groups with interesting finiteness properties. We say that a group $G$ is of finiteness type $\mathcal{F}_k$ if it admits a $K(G,1)$ which is a CW-complex with finitely many cells of dimension $\leq k$. Historically, the first examples of groups of type $\mathcal{F}_{k-1}$ and not $\mathcal{F}_{k}$ for all $k\geq 3$, were the Stallings--Bieri groups $SB_k$ which arise as kernels $SB_k={\rm ker } (F_2 \times \dots\times F_2\to \ZZ)$ of homomorphisms from products of $k$ free groups onto $\ZZ$ \cite{Sta-63,Bie-76}. Determining their precise Dehn functions turned out to be a challenging problem \cite{Ger-95,BauBriMilSho-97,Bri-99,DisEldRilYou-09}
 that was only resolved in full generality very recently by Carter and Forester \cite{CarFor-17}. 
 
 These results raise the question if one can also understand the Dehn functions and finiteness properties of more general subgroups $K\leq G_1\times \dots \times G_n$ of direct products of groups and in particular of {\emph{coabelian subgroups (of corank $l$)} arising as kernel of a surjective homomorphism $\phi: G_1\times \dots \times G_n\to \ZZ^l$. The finiteness properties of such subgroups have been studied by various authors \cite{BriHowMilSho-02, BriHowMilSho-13, Koc-10, Kuc-14}. They are now well-understood when the $G_i$ are non-abelian limit groups and $K$ is coabelian. A first general study of Dehn functions of subgroups of direct products of groups was performed by Bridson \cite{Bri-01} in the cocyclic case and by Dison \cite{Dis-08} in the coabelian case. Other results in this area include the ones mentioned in the previous paragraph, as well as \cite{LloTes-18-II}. However, we are still far from a complete description of the Dehn functions that can arise, even if we assume that the $G_i$ are free groups and that $K$ is coabelian. Indeed, the only such examples for which the precise Dehn function is known are the ones which are virtually a direct product of Stallings--Bieri groups and free groups.

The goal of this work is to develop new methods for computing the precise Dehn functions of coabelian subgroups of direct products of groups in terms of the Dehn functions of the factors. To obtain them we translate the main result of \cite{CarFor-17} to an algebraic setting and then generalise it in two different ways. We will then provide several applications of our methods to subgroups of direct products of free groups (SPFs), to 1-ended irreducible groups with interesting finiteness properties, and to subgroups of direct products of right-angled Artin groups. The application to SPFs is one of the main motivations for our work. In particular, we can now determine the precise Dehn functions for a large family of such groups considered in \cite{Dis-08}. 

\subsection{Dehn functions of coabelian subgroups of direct products of groups}
In their work Carter and Forester \cite{CarFor-17} provide upper bounds on Dehn functions of level sets of $\RR$-valued height functions on direct products of simply connected cube complexes. To obtain their bounds they introduce a cell structure induced by subdivision of the cube complex along the level sets of the height function. They then deduce area bounds in terms of the combinatorial area of edge loops in this cell structure. Related approaches have been used for estimating filling functions in horospheres in symmetric spaces with a product structure (see Gromov \cite{Gro-93} and Drutu \cite{Dru-04}).

We generalise the methods in \cite{CarFor-17} and translate them into an algebraic setting. Roughly speaking our generalisation consists of replacing height maps to $\RR$ by ``multi-dimensional'' height maps to $\RR^m$ which admit a splitting. In algebraic terms this equates to split epimorphisms onto $\ZZ^m$. This formulation in algebraic terms allows us to prove our results for arbitrary groups that do not need to admit an action on a simply connected cube complex. As a consequence we can compute the precise Dehn function of many kernels of homomorphisms from direct products of at least three groups onto free abelian groups.
\begin{restatable*}{theorem}{algebraictriangle}\label{thm:algebraictriangle}
	For $1\leq i \leq 3$ let $G_i$ be finitely presented groups and let $1\to N_i\to G_i\xrightarrow{\phi_i} \ZZ^m \to 1$ be right-split short exact sequences. 
	Let $\phi\colon G_1\times G_2\times G_3\to \ZZ^m$ be defined by $\phi(g_1, g_2, g_3) = \sum_{i=1}^3 \phi_i(g_i)$. Let $\overline{f}$ be the superadditive closure of the Dehn function $f$ of $G_1\times G_2\times G_3$.
	Then $K:=\ker(\phi)$ is finitely presented and its Dehn function satisfies $f(n)\preccurlyeq \delta_{K}(n)\preccurlyeq \overline{f}(n) \cdot log(n)$. If, moreover, $\frac{f(n)}{n}$ is superadditive then $f(n)\asymp \delta_K(n)$.
\end{restatable*}
Here the \emph{superadditive closure} of $f:\NN_{>0}\to \RR_{>0}$ is the smallest function $\overline{f}(n)$, which is bounded below by $f(n)$ and satisfies $\overline{f}(n+m)\geq \overline{f}(n)+\overline{f}(m)$ for $n,m >0$.  The moreover-part of our result is particularly interesting, because we are not aware of any non-linear Dehn function of a finitely presented group for which $f(n)/n$ is not superadditive (see also \cite{GubSap-99}). In particular, every function of the form $n\mapsto n^{a}$, with $a\geq 2$, and $n \mapsto e^{n}$ has this property.

As a direct consequence we obtain the following improvement of a result of Bridson:
\begin{corollary}\label{cor:bridson}
 Let $\phi: G_1\times G_2\times G_3\to \ZZ$ be a homomorphism whose restriction to each of the $G_i$ is surjective and let $f(n)$ be the Dehn function of $G_1\times G_2\times G_3$. Then the Dehn function of $K:=\ker(\phi)$ satisfies $f(n)\preccurlyeq \delta_K(n)\preccurlyeq \overline{f}(n)\cdot log(n)$.
\end{corollary}
Bridson \cite[Theorem 0.3]{Bri-01} showed that $\delta_K(n)\preccurlyeq n\cdot \delta_{G_1\times G_2}(n)+\delta_{G_3}(n)$.

Replacing splittings by a more general concept, which we call $P$-splitting (see \cref{def:Psplit}), allows us to prove a variation of Theorem \ref{thm:algebraictriangle}.
\begin{restatable*}{theorem}{algebraicsquare}\label{thm:algebraicsquare}
	Let $G_1, \dots, G_4$ be finitely presented groups and let $1\to N_i\to G_i\xrightarrow{\phi_i} \ZZ^m\to 1$ be short exact sequences. Suppose that there is a $P$ such that each of these sequences $P$-splits. 
	Let $\phi\colon G_1\times G_2\times G_3\times G_4\to \ZZ^m$ be defined by $\phi(g_1,g_2,g_3,g_4)=\sum_{i=1}^4 \phi_i(g_i)$ and let $f(n)$ be the Dehn function of $G_1\times \dots \times G_4$. 
	Then $K = \ker(\phi)$ is finitely presented and its Dehn function satisfies $f(n)\preccurlyeq \delta_K(n)\preccurlyeq \overline{f}(n)\cdot log(n)$. 
	If, moreover, $\frac{f(n)}{n}$ is superadditive then $f(n)\asymp \delta_K(n)$.
\end{restatable*}

In a similar vein to \Cref{cor:bridson}, we can now make a statement about kernels of maps to $\ZZ^2$. 

\begin{corollary}
	Let $\phi: G_1\times G_2\times G_3\times G_4\to \ZZ^2$ be a homomorphism whose restriction to each of the $G_i$ is surjective and let $f(n)$ be the Dehn function of $G_1\times G_2\times G_3\times G_4$. Then the Dehn function of $K:=\ker(\phi)$ satisfies $f(n)\preccurlyeq \delta_K(n)\preccurlyeq \overline{f}(n)\cdot log(n)$.
\end{corollary}

The number of factors here cannot be reduced. Indeed, the kernel of the homomorphism $\phi\colon F_2\times F_2\times F_2\to \ZZ^2$ defined by abelianization on the factors satisfies a cubic lower bound on its Dehn function \cite{Dis-09}.

\begin{remark}
	Our proofs provide a constructive way of filling a loop with a disk. One can apply a similar construction for the homological finite presentations from \cite{BKS} to bound the homological Dehn function when each $G_i$ is of type $FP_2$. 
	More precisely, if in \cref{thm:algebraictriangle} or \cref{thm:algebraicsquare} we assume that each $G_i$ is of type $FP_2$ (rather than finitely presented), then $K$ is of type $FP_2$ and $f(n)\preccurlyeq \delta_K(n)\preccurlyeq f(n)\cdot log(n)$, where $f$ is the homological Dehn function of $\prod_i G_i$. For the upper bound one uses that homological Dehn functions are always superadditive \cite[Proposition 2.20]{BKS}.
\end{remark}

\begin{remark}\label{rem:Disons-Thm-11-3-4}
As a further application of Theorem \ref{thm:algebraicsquare}, we will improve a result of Dison \cite[Theorem 11.3 (4)]{Dis-08}, which provides upper bounds on Dehn functions of kernels of certain homomorphisms from direct products of finitely presented groups onto $\ZZ^l$ (see Theorem \ref{thm:Disons-Thm-11-3-4} for a precise statement). In particular, we can compute their precise Dehn functions if $f(n)/n$ is superadditive, where $f(n)$ is defined as in Theorem \ref{thm:algebraicsquare}.
\end{remark}

We will provide several applications of Theorems \ref{thm:algebraictriangle} and \ref{thm:algebraicsquare}. They will also show that the two results have partially complementary applications.

\subsection{Dehn functions of subgroups of products of free groups}
Our first application is to coabelian subgroups of direct products of free groups. 
They provide a natural generalisation of the Stallings--Bieri groups. With the exception of groups that are virtually direct products of free groups, all SPFs have interesting finiteness properties \cite{BriHowMilSho-02, BriHowMilSho-13}. This makes it natural to pose the question how wild the Dehn functions of SPFs can be. It was first raised by Dison:
\begin{question}[{Dison \cite[Question 1]{Dis-08}}]
\label{qn:Dison1}
 Does every finitely presented SPF have a polynomially bounded Dehn function?
\end{question} 
Dison also posed the stronger version of Question \ref{qn:Dison1} whether the class of Dehn functions of SPFs satisfies a uniform polynomial upper bound. While Tessera and the second author gave a negative answer to the uniform version of Dison's question, their examples are not (virtually) coabelian \cite{LloTes-18-II}. This raises the following question:
\begin{question}[{\cite[Question 4]{LloTes-18-II}}]
\label{qn:Dison2}
 Is there a polynomial $p(n)$ such that $\delta_G(n)\preccurlyeq p(n)$ for all coabelian SPFs $G$?
\end{question}
A natural approach to these questions is to develop methods that provide us with good upper bounds on the Dehn functions of SPFs. This approach was pursued by Dison \cite{Dis-08}. Roughly speaking he showed that coabelian SPFs that satisfy strong enough finiteness properties admit polynomially bounded Dehn functions, providing evidence towards a positive answer to Question \ref{qn:Dison1}. Moreover, Dison showed that under even stronger assumptions one can obtain a uniform polynomial bound. 

More precisely, for $m\geq l$ and $r\geq 3$, Dison defines the group $K_m^r(l)$ as kernel of a homomorphism $\phi: F_m^{\times r}=F_m\times \dots \times F_m\to \ZZ^l$ where every factor surjects onto $\ZZ^l$ and proves that $\delta_{K_m^r(l)}(n)\preccurlyeq n^5$ if $2l\leq r$. While a priori there are different homomorphisms with this property, they are the same up to an automorphism of $F_m^{\times r}$; in particular, up to this automorphism, $K_m^r(l)$ does not depend on the choice of homomorphism $\phi$ \cite[Lemma 13.1]{Dis-08}, explaining the notation.

As an application of Theorem \ref{thm:algebraicsquare} we can compute the precise Dehn functions for this family of examples. 
\begin{theorem}
 \label{thm:Dehn-fct-Disons-groups}
 For $r\geq 4$ and $\left\lceil \frac{l}{2}\right\rceil \leq \frac{r}{4}$ the group $K_m^r(l)$ has quadratic Dehn function. 
\end{theorem}
Note that this result is optimal in $r$ for the subfamily $K_2^r(2)$, $r\geq 3$, since $K_2^3(2)$ satisfies a cubic lower bound on its Dehn function \cite{Dis-09}. Theorem \ref{thm:Dehn-fct-Disons-groups} will follow from the computation of the precise Dehn function for a more general class of coabelian SPFs (see Theorem \ref{thm:Fmcoabelian}). 

While our results do not extend the family of SPFs for which we know that their Dehn function satisfies a uniform polynomial upper bound, they do suggest that the current upper bounds on Dehn functions of SPFs in the literature might still be far from optimal. This provides further impetus to trying to improve these bounds and thereby to making further progress on Questions \ref{qn:Dison1} and \ref{qn:Dison2}. Moreover, our result also significantly extends the class of SPFs for which we know their precise Dehn function.

\subsection{Groups with interesting finiteness properties and prescribed Dehn function}
We call an infinite group $G$ \emph{irreducible}, if it does not have a finite index subgroup of the form $H_1\times H_2$ with $H_1$ and $H_2$ infinite.

In \cite[Corollary 1.1]{CarFor-17} Carter and Forester observe that a consequence of their results is the existence of groups of type $\mathcal{F}_{n-1}$ and not $\mathcal{F}_n$ with quadratic Dehn function. Finiteness properties are preserved when taking free products and direct products with groups of type $\mathcal{F}_{\infty}$. Putting these two results together shows that for every function $f(n)\succcurlyeq n^2$ which can be realised as Dehn function of a group $G$ of type $\mathcal{F}_{\infty}$ there is a group of type $\mathcal{F}_{n-1}$ and not $\mathcal{F}_n$ with the same Dehn function. Indeed, the direct product $G\times \mathrm{SB}_n$ has Dehn function $f(n)$. We also note that the free product $G\ast \mathrm{SB}_n$ has Dehn function $\overline{f}(n)$. Observe that the groups obtained via direct products are 1-ended, but not irreducible, while the groups obtained via free products are irreducible, but not 1-ended.

As a consequence one obtains that for non-hyperbolic groups Dehn functions do not impose any restrictions on their finiteness properties. A priori this might change if one imposes additional group theoretic constraints. Here we show: 
\begin{restatable*}{theorem}{Dehnfinprops}\label{thm:Dehn-fin-props}
Let $f(n)\succcurlyeq n^2$ be a function which can be realised as Dehn function of a group $G$ of type $\mathcal{F}_{\infty}$ and let $k\geq 3$. Then there is a 1-ended irreducible group $K$ of type $\mathcal{F}_{k-1}$ and not $\mathcal{F}_k$ whose Dehn function satisfies $\overline{f}(n)\preccurlyeq \delta_K(n)\preccurlyeq log(n)\cdot \overline{f}(n)$ for $\overline{f}$ the superadditive closure of $f$.
If, moreover, $\frac{f(n)}{n}$ is superadditive then $f(n)\asymp \delta_K(n)$.
\end{restatable*}
It is worth recalling here that to our knowledge all known Dehn functions with $f(n)\succcurlyeq n^2$ satisfy that $\frac{f(n)}{n}$ is superadditive and can be realised as Dehn function of a group of type $\mathcal{F}_{\infty}$. Thus, in all known cases the assumption that the group is simultaneously 1-ended and irreducible does not impose any additional constraints.

\subsection*{Notation} For monotonely increasing functions $f,g: \RR_{>0}\to \RR$ we write $f\preccurlyeq g$ if there is a constant $C>0$ such that $f(n)\leq Cg(Cn + C)+Cn+C$ for all $n\in \RR_{>0}$ and $f\asymp g$ if $f\preccurlyeq g \preccurlyeq f$. If $f\asymp g$ we say that $f$ and $g$ are asymptotically equivalent.

Throughout this article by a word in a set $X$ we will mean an element of the free monoid with generating set $X\sqcup X^{-1}$. For a group $G$ and a subset $X\subset G$, we say that $h\in G$ is obtained from $g\in G$ by applying elements of $X$ to $g$ if there is a word $w(X)$ in $X$ such that $h=g\cdot w(X)$ in $G$; equivalently $h\in g\cdot \langle X\rangle$.

For a word $w(X)=x_1\cdot \dots \cdot x_n$ with $x_i\in X\sqcup X^{-1}$ we will denote by $|w|:=n$ its \emph{word length}. Moreover, for  group elements $g, h\in G=\left\langle X\right\rangle$ we will denote by $d(g,h)(=d_{G,X}(g,h))$ their distance in the \emph{word metric} on the Cayley graph $\mathrm{Cay}(G,X)$ of $G$ with respect to the generating set $X$.

\subsection*{Structure} In Section \ref{sec:Background} we provide some background on Dehn functions, superadditivity and filling pairs. In Section \ref{sec:trianglemethod} we prove Theorem \ref{thm:algebraictriangle}. In Section \ref{sec:square} we prove Theorem \ref{thm:algebraicsquare}. In Section \ref{sec:SPFs-with-quadratic-Dehn} we apply Theorem \ref{thm:algebraicsquare} to study Dehn functions of SPFs and prove Theorem \ref{thm:Dehn-fct-Disons-groups}. The proof of Theorem \ref{thm:Dehn-fin-props} is contained in Section \ref{sec:finprops-Dehn-fcts}. Finally, in Section \ref{sec:raags}, we will provide some straight-forward applications of our results to Dehn functions of coabelian subgroups of right-angled Artin groups. 

\subsection*{Acknowledgements}
The authors would like to thank Rob Merrell and an anonymous referee for their helpful comments and suggestions.

\section{Dehn functions, superadditivity and filling pairs}
\label{sec:Background}

\begin{definition}
	Let $G$ be a group given by a finite presentation $\langle X\mid R\rangle$. 
	Let $w$ be a word in $X$ which represents the trivial element of $G$. 
	We define the {\em area} of $w$ as: 
	$$\Area(w) = \min\{l\mid w = \prod_{i=1}^l w_ir_iw_i^{-1}, w_i\in F(X), r_i\in R^{\pm1}\}.$$
	We define the {\em Dehn function} of $G$, $\delta_G:\NN \to \NN$ by
	$$\delta_G(n) = \max\{\Area(w) \mid |w|\leq n\}.$$
\end{definition}

A priori the definition depends on the choice of finite presentation for $G$. However, it is well-known that its asymptotic equivalence class does not. This explains the omission of the finite presentation in the notation $\delta_G$ for the Dehn function.

Recall the following well-known result (see \cite{Bric-93}). 
\begin{lemma}\label{lem:retraction}
	Let $G, H$ be finitely presented groups. 
	Suppose that $G\to H$ is a retraction. 
	Then $\delta_H\preccurlyeq \delta_G$. 
\end{lemma}

Throughout we will be interested in Dehn functions of products of groups. They are described by the following

\begin{lemma}
\label{lem:Dehndirectproduct}
	For $k\geq 2$ let $G_1, \dots, G_k$ be infinite finitely presented groups and denote by $\delta_i$ the Dehn function of $G_i$. 
	Then the Dehn function of the product $G = \prod_i G_i$ is equivalent to $\max\left\{n^2, \delta_i, 1\leq i \leq k\right\}$. 
\end{lemma}
\begin{proof}
	There are retractions $G\to G_i$ for each $i$, thus $\delta_G$ is bounded below by $\delta_i$. 
	Since the $G_i$ are infinite, $G$ is not hyperbolic and thus $\delta_G$ is also bounded below by $n^2$. For the upper bound see \cite[Proposition 2.1]{Bric-93}. 
\end{proof}

Various of the arguments used will require superadditivity of functions: 
\begin{definition}
	A function $f\colon \NN\to \RR_{>0}$ is called {\em superadditive} if $f(n+m)\geq f(m) + f(n)$. 
	The {\em superadditive closure} $\overline{f}$ of $f$ is the smallest superadditive function with $\overline{f}(n)\geq f(n)$. 
\end{definition}
It is currently an open conjecture that every Dehn function is equivalent to its superadditive closure \cite[Conjecture 1]{GubSap-99}. The following straight-forward observation provides evidence:
\begin{lemma}\label{lem:Quotient-condition-superadd}
	Let $f\colon \NN \to \RR_{>0}$ be a function that satisfies that $\frac{f(n)}{n}$ is non-decreasing in $n$. Then $f$ is superadditive.
\end{lemma}
Here we will be interested in a stronger condition than superadditivity. Namely, we will be interested in non-decreasing functions $f\colon \NN \to \RR_{>0}$ with the property that $n\mapsto \frac{f(n)}{n}$ is superadditive. Note that neither superadditivity nor this property need to be preserved by asymptotic equivalence of functions. However, we have:
\begin{lemma}\label{lem:nice-representatives}
Let $f:\NN \to \RR_{>0}$ be non-decreasing such that $f\asymp g\succcurlyeq n^2$ for a function $g:\NN \to \RR_{>0}$ with $g'(n):=\frac{g(n)}{n}$ superadditive. Let $\overline{f}'(n)$ be the superadditive closure of $f'(n):= \frac{f(n)}{n}$ and let $\widehat{f}(n):= n\cdot \overline{f}'(n)$. Then $\widehat{f}\asymp g$ and $\widehat{f}\geq f$.
\end{lemma}
\begin{proof}
By definition $\widehat{f}(n)\geq f(n)$ and thus $\widehat{f}\succcurlyeq g$. Conversely, $g\succcurlyeq f$ implies that there is $C\geq 1$ such that $f(n)\leq Cg(Cn+C)+Cn+C$ for all $n\in \NN$. Let $n\in \NN$ and let $n=n_1+\dots + n_k$ be a partition realising $\overline{f}'(n)$. Then:
\begin{align*}
\widehat{f}(n) = n \cdot \sum_{i=1}^k f'(n_i)&\leq n\cdot \sum_{i=1}^k \frac{C g(Cn_i+C)+Cn_i+C}{n_i}\\
&= n\cdot \sum_{i=1}^k \frac{C (Cn_i+C) g'(Cn_i+C)+Cn_i+C}{n_i}\\
&\leq C\cdot n \cdot \left(\sum_{i=1}^k 2C g'(C n_i+C)\right) + n \cdot (2 C k)\\
&\leq 2C^2\cdot n \cdot g'(C(n+k)) + n \cdot (2 C k)\\
&\leq 2C g(n+k) + 2C (n+k)^2\\
&\leq 2C g(2n)+2C (2n)^2.
\end{align*}
Since $g\succcurlyeq n^2$, this completes the proof.
\end{proof}
The following is an immediate consequence of the observation that superadditivity of $f(n)/n$ implies superadditivity of $f$.
\begin{remark}\label{rem:superadditivity}
Let $f:\NN\to \RR_{>0}$ be a non-decreasing function. Let $\overline{f}$ be its superadditive closure and let $g(n):=n\cdot \overline{f}'(n)$ for $\overline{f}'(n)$ the superadditive closure of $f(n)/n$. Then $f\leq \overline{f}\leq g$.
\end{remark}

Our methods will also provide upper bounds on the filling diameters of our fillings and thus provide us with filling pairs. Here we summarize the properties of filling diameters and pairs that we will require.

We define the \textit{(extrinsic) filling diameter} of a filling $w(X) = \prod_{i=1}^n v_i(X)\cdot r_i^{\pm 1}\cdot v_i(X)^{-1}$ of a null-homotopic word $w(X)$ in $G=\langle X\mid R\rangle$ as the maximal distance $\max_{1\leq i\leq n}\left\{d_G(1,v_i(X))  \right\}$ of the conjugators from $1\in G$ in $Cay(G,X)$. 

For functions $f,g: \NN \to \RR>0$, we call $(f,g)$ an \textit{(extrinsic) filling pair} for $G$ if every null-homotopic word $w(X)$ of length $\leq n$ admits a filling of area $\leq f(n)$ and diameter $\leq g(n)$. The following is well-known and it follows easily from the proof of the upper bound on the Dehn function of a direct product of groups given in \cite{Bric-93}.
\begin{lemma}\label{lem:filling-pair-product}
 Let $G_1,\dots,G_k$ be finitely presented groups with filling pairs $\left(f_i,g_i\right)$. Then $$(f(n),g(n)):=\left(n^2+\max_{1\leq i\leq k}\{ f_i(n)\},n+\max_{1\leq i\leq k} \{g_i(n)\}\right)$$ is a filling pair for $\prod_{i=1}^k G_i$.
\end{lemma}

\begin{remark}
One can also define the \emph{intrinsic filling diameter} of a filling of a null-homotopic word, as the diameter of a corresponding van Kampen diagram, and correspondingly define intrinsic filling pairs. Using that in \cite{Bric-93} Brick constructs explicit van Kampen diagrams satisfying the asserted upper area bounds, one observes that Lemma \ref{lem:filling-pair-product} also holds for intrinsic filling pairs. In particular, one readily checks that all of our results about extrinsic filling pairs remain true for intrinsic filling pairs by carefully going through their proofs and observing that the fillings we construct are obtained by gluing together van Kampen diagrams whose distance to the base point is bounded by a uniform multiple of the length of the word we fill. However, note that in general the functions $EDiam(n)$, resp. $IDiam(n)$, defined as the maximum over all minimal extrinsic, resp. intrinsic, filling diameters of null-homotopic words of length $\leq n$ are not equivalent \cite{BriRil-09, BRS-07}. For simplicity in the sequel we will thus stick to extrinsic filling diameters and pairs. In particular, the terms ``filling diameter'' and ``filling pair'' will always be referring to their extrinsic version.
\end{remark}

\section{Algebraic triangle method}
\label{sec:trianglemethod}

In \cite{CarFor-17}, a method is developed to study the Dehn function of a kernel of a homomorphism $G_1\times G_2\times G_3\to \ZZ$. 
For three cube complexes $X_i$ equipped with height functions $h_i\colon X_i\to \RR$, that is, functions that restrict to linear maps on cubes and map edges onto intervals of the form $\left[n,n+1\right]$, it provides an upper bound on the Dehn function of the zero-level set of the sum of the $h_i$.
\begin{theorem}[{\cite[Theorem 4.2]{CarFor-17}}]
\label{thm:Carter-Forester}
	 Suppose $a\geq 2$ and let $X_1, X_2$ and $X_3$ be simply connected cube complexes with height functions $h_i:X_i\to \RR$ such that each $X_i$ is admissible and has Dehn function $\preccurlyeq n^a$. 
	 Then the zero level set $[X_1 \times X_2 \times X_3]_0$ of $h=\sum_{i=1}^3 h_i$ is simply connected and has Dehn function $\preccurlyeq n^a$.
\end{theorem}
Here a cube complex $X_i$ is {\em admissible} if each vertex is contained in a monotone line, where a {\em monotone line} is a subcomplex $L_i$ of $X_i$ such that $h_i|_{L_i}$ is a homeomorphism.  
The proof in \cite{CarFor-17} exploits that the cubical structure on the factors induces a sliced cell structure on certain level sets, enabling the authors to derive upper bounds on the Dehn function via area estimates with respect to a cellular structure.

The first goal of this work is a generalisation of this result. A natural strategy for obtaining such an extension is to generalise the notion of height functions. A first step in this direction would be to consider maps $h_i: X_i\to \RR^m$ which restrict to linear maps on cubes, map vertices of $X_i$ into the integer lattice of $\ZZ^m$ and are Lipschitz. This approach has the advantage that one can still exploit the combinatorial structure of the cube complex by inducing a generalised version of the sliced cell-structures used in \cite{CarFor-17} on the level sets of $h$. However, it turns out that one can evade the use of a combinatorial structure altogether and thereby even drop the condition that the $X_i$ are cube complexes. 
Geometrically one can do so by extending the definition of height function to Lipschitz maps $h_i\colon X_i\to \RR^m$ for some $n$, where the $X_i$ are suitable locally compact and locally connected length spaces (for instance direct products of Cayley graphs of finitely presented groups). 
One can then define monotone planes to be subspaces of the $X_i$ such that $h_i$ restricts to a bilipschitz homeomorphism onto $\RR^m$ and formulate a result similar to Theorem \ref{thm:Carter-Forester} in this setting.

However, we shall pursue a different purely algebraic approach which avoids some of the geometric subtleties that one encounters in the aforementioned approach via Lipschitz height maps.

\algebraictriangle

While the condition that $\frac{f(n)}{n}$ is superadditive might seem restrictive on first sight, we are not aware of any non-hyperbolic group that does not satisfy it. Lemma \ref{lem:Quotient-condition-superadd} implies:
\begin{corollary}\label{cor:triangle-thm-1}
	If in Theorem \ref{thm:algebraictriangle} $\frac{f(n)}{n^2}$ is non-decreasing, then $\ker(\phi)$ is finitely presented and has Dehn function $f$.
\end{corollary}
We also obtain the following estimate on filling pairs:
\begin{corollary}\label{cor:triangle-thm-3}
If $(f_i,g_i)$ is a filling pair for $G_i$ then the proof of Theorem \ref{thm:algebraictriangle} shows that $K$ admits a filling pair of the form $( log(n)\cdot \overline{f}(n),g(n))$, with $\overline{f}$ the super-additive closure of $f(n)\asymp n^2+\sum_{i=1}^3 f_i(n)$ and $g(n)\asymp n+ \sum_{i=1}^3 g_i(n)$.
\end{corollary}

We will now fix specific generating sets for the groups $G_i$ which are compatible with the splittings $s_i$ as follows. Fix generating sets $Z = \{z_1, \dots, z_m\}$ for $\ZZ^m$ and $V_i$ for $G_i$. Let $Z_i = \{s_i(z_1), \dots, s_i(z_m)\}$ and $Y_i = \{v s_i(\phi_i(v))^{-1}\mid v\in V_i\} \subset \ker(\phi_i)$. Then $Y_i\cup Z_i$ is a generating set for $G_i$. 

In the sequel we will use various sets that generate subgroups of $G_1\times G_2\times G_3$ defined in terms of these generating sets as laid out in the following lemma. 

\begin{lemma}\label{lem:trianglegensets}
	Let $G_i, Y_i, Z_i, Z$ and $\phi$ be as above. 
	Define the following subsets of $G_1\times G_2\times G_3$.
	\begin{itemize}
		\item  $T_1 = \{(s_1(z), s_2(z)^{-1}, e)\mid z\in Z\}$,
		\item  $T_2 = \{(e, s_2(z), s_3(z)^{-1})\mid z\in Z\}$,
		\item  $T_3 = \{(s_1(z)^{-1}, e, s_3(z))\mid z\in Z\}$,
		\item  $U_1 = \{(y, e, e)\mid y\in Y_1\}$, 
		\item  $U_2 = \{(e, y, e)\mid y\in Y_2\}$, 
		\item  $U_3 = \{(e, e, y)\mid y\in Y_3\}$.
	\end{itemize}
	Let $T = \cup_i T_i$ and $U = \cup_i U_i$. 
	Then the following hold: 
	\begin{itemize}
		\item $\langle U\cup T\rangle  = \ker(\phi)$,
		\item $\langle T\rangle\cong \ZZ^{2m}$,
		\item $\langle U_i\cup T_i\rangle  \cong G_i$,
		\item $\langle U_i\cup T\rangle  \cong G_i\times \ZZ^m$,
		\item $\langle U_i\cup U_j \cup T\rangle  \cong G_i\times G_j$, for $i\neq j$. 
	\end{itemize}
\end{lemma}
\begin{proof}
	We will prove the first and last statements. The others can be deduced via similar reasoning.
	Let $K \vcentcolon= \ker(\phi)$. 
	We will begin with the first statement. 
	It is clear that $U\cup T\subset K$. 
	Thus we must prove that any element of $K$ can be written as a word in these generators. 
	Let $(g_1, g_2, g_3)\in K$. 
	Modulo an element of $\langle U_1\cup U_2\cup T_2\cup T_3\rangle$ we can reduce to an element of the form $(e, e, h)$. 
	Applying elements from $U_3\cup T_2$, we can obtain an element of the form $(e, \left(s_2(\phi_3(h))\right)^{-1}, e)$. 
	However, since $(e, e, h)\in K$ we see that $\phi_3(h) = 0$ and we are done. 
	
	For the last statement we may assume $i = 1, j = 2$. 
	Let $H = \langle U_1 \cup U_2\cup T\rangle$. 
	Consider the projection $H\to G_1\times G_2$. 
	It is easy to see that it is surjective. Indeed, for $(g_1,g_2)\in G_1\times G_2$ there is an element of $\langle U_1\cup U_2\cup T_2\cup T_3\rangle$ of the form $(g_1, g_2, s_3(\phi_1(g_1)+\phi_2(g_2))^{-1})$. 
	However, any two such elements of $H$ differ by an element of the form $(e,e,s_3(z))\in K$ for some $z\in \ZZ^m$. This implies that the projection map is injective and thus an isomorphism.
\end{proof}

We fix finite presentations for all groups in Lemma \ref{lem:trianglegensets} (except $\ker(\phi)$) and for $G_1\times G_2\times G_3$ with respect to the given generating sets. By Lemmas \ref{lem:Dehndirectproduct} and \ref{lem:nice-representatives} we may assume that the Dehn functions of all of the groups $\ZZ^{2m}$, $G_i\times \ZZ^m$, $G_i\times G_j$ and $G_1\times G_2\times G_3$ are bounded above by the function $f$ from Theorem \ref{thm:algebraictriangle} (after possibly making $f$ larger in its equivalence class).

We now provide a filling of a null-homotopic word in the generating set $U\cup T$ of $K$. We first construct fillings for triangles spanned by 3 elements of $G_1\times G_2 \times G_3$ as in Figure \ref{fig:algebraicspanningtriangle} by gluing together fillings in subgroups generated by subsets of $U\cup T$ as in Lemma \ref{lem:trianglegensets}. We then fill an arbitrary loop by tiling it by triangles as in Figure \ref{fig:hyptriangle}. The connection to the geometric approach from \cite{CarFor-17} is that if there are geometric actions of $G_i\curvearrowright X_i$ on length spaces and admissible lines $L_i$ through every point in $X_i$, then our subgroups act geometrically on 0-level sets of products of the $X_i$ and the $L_i$. 
	
\subsection{A spanning triangle}

Throughout this section $K$ will always be equipped with the generating set $U\cup T$. 
Let $\a = (\a_1, \a_2, \a_3), \b = (\b_1, \b_2, \b_3)$ and $\c = (\c_1, \c_2, \c_3)$ be elements of $K$. 
We will construct a triangular loop in the Cayley graph of $K$ with these three elements as vertices. It will arise as boundary loop of a subdivided triangle as in Figure \ref{fig:algebraicspanningtriangle} whose edges will be labelled by words in the generating sets of the adjacent bounded regions. Where two bounded regions are adjacent to an edge, note that one generating set is contained in the other and the edge will be labelled by a word in the smaller one of the generating sets. Throughout this section by an edge we will be referring to an edge in the 1-skeleton of the triangle (as opposed to an edge of the Cayley graph).

As a consequence we will be able to control the filling area of the boundary word of the triangle in terms of filling areas of words in the subgroups labelling the bounded regions. By carefully controlling the lengths of the words labelling each of the edges throughout our construction, we will thus be able to obtain bounds on the area of the boundary word of the triangle in terms of the Dehn functions of the groups labelling the bounded regions and thus in terms of $f$.

We will now describe the construction of the words labelling the edges of the triangle. We keep track of upper bounds on their lengths in Figure \ref{fig:lengthtri}. Note that due to symmetries it will be sufficient to focus on the edges $\circled{1}$ - $\circled{4}$. 

We start by constructing a path in the Cayley graph labelling the 3 edges between $\a$ and $\b$. 

First, by applying generators from $U_2\cup T$, we can construct a path from $(\a_1,\a_2,\a_3)$ to a vertex of the form $(\a_1', \b_2, \a_3)$, where $\a_1$ and $\a_1'$ differ by an element of $s_1(\ZZ^m)$. The corresponding word in $U_2\cup T$ provides the label for the edge $\circled{1}$.

Next we use generators from $U_3\cup T$ to construct a path from the element $(\a_1', \b_2, \a_3)$ to an element of the form $(\a_1, \b_2', \b_3)$; 
this is possible since $\a_1\cdot (\a_1')^{-1}\in s_1(\ZZ^m)$. Similar as above, we see that $\b_2\cdot (\b_2')^{-1}\in s_2(\ZZ^m)$. 

Finally, since $\b_2\cdot (\b_2')^{-1}\in s_2(\ZZ^m)$, we can construct a path from $(\a_1, \b_2', \b_3)$ to $(\b_1, \b_2, \b_3)$ using elements from $U_1\cup T$. 

A similar construction provides paths labelling the edges from $\a$ to $\c$ and from $\b$ to $\c$. 

We now describe the construction of the three interior vertices in \Cref{fig:algebraicspanningtriangle} and the corresponding edges.

Applying generators from $U_3\cup T$ to $(\a_1', \b_2, \a_3)$ we can construct a path to a vertex of the form $(\a_1''', \b_2, \c_3)$. 
As before $\a_1'(\a_1''')^{-1}\in s_1(\ZZ^m)$. Similarly we can apply generators from $U_2\cup T$ to the element $(\a_1'', \a_2, \c_3)$ to construct a path to a vertex of the form $(\overline{\a}_1, \b_2, \c_3)$ with $\a_1'' \cdot \overline{\a}_1^{-1}\in s_1(\ZZ^m)$. A priori we could have $\a_1'''\neq \overline{\a}_1$ in which case we could not close the loop. However, recalling that $\a_1 (\a_1')^{-1}, \a_1 (\a_1'')^{-1}\in s_1(\ZZ^m)\cong \ZZ^m$, we obtain $\a_1'''\cdot \overline{\a}_1^{-1}\in s_1(\ZZ^m)$ and thus $(\a_1'''\cdot \overline{\a}_1^{-1},e,e)= (\a_1''',\b_2,\c_3)\cdot(\overline{\a}_1,\b_2,\c_3)^{-1}\in K\cap s_1(\ZZ^m)=\left\{(e,e,e)\right\}$. 

Similar arguments allow us to obtain vertices $(\a_1, \b_2''', \c_3)$ and $(\a_1, \b_2, \c_3''')$ with $\b_2\cdot (\b_2''')^{-1}\in s_2(\ZZ)$ and $\c_3\cdot (\c_3''')^{-1}\in s_3(\ZZ^m)$ together with words labelling the corresponding edges.

Finally, we can construct a path from $(\a_1''',\b_2,\c_3)$ to a vertex of the form $(\a_1,\overline{\b}_2,\c_3)$ labelled by a word in the elements of $T$. In particular $\b_2\cdot \overline{\b}_2^{-1}\in s_2(\ZZ^m)$, which implies that also $\overline{\b}_2\cdot (\b_2''')^{-1}\in s_2(\ZZ^m)$. Hence, $(e,\overline{\b}_2\cdot (\b_2''')^{-1},e)=(\a_1,\overline{\b}_2,\c_3)\cdot (\a_1,\b_2''',\c_3)^{-1}\in K\cap s_2(\ZZ^m)=\left\{(e,e,e)\right\}$ and therefore $(\a_1,\overline{\b}_2,\c_3)=(\a_1,\b_3''',\c_3)$. This provides a word in $T$ labelling the edge $\circled{4}$. 

Analogous arguments allow us to construct words in T labelling the other two edges of the interior triangle.

It remains to bound above the lengths of the edges labelled $\circled{1} - \circled{4}$ with the other bounds following by symmetries. 

Denote by $d_i$ the word metric on $G_i$ with respect to the generating set $Y_i\cup Z_i$ and by $d_{\ZZ^m}$ the word metric on $\ZZ^m$ with respect to the generating set $Z$. Note that the bijection $Z\to Z_i$ of generating sets induces an isometric embedding $s_i(\ZZ^m)\leq G_i$ and that the projection $\phi_i: G_i\to \ZZ^m$ is length non-increasing. We will frequently use this without further mention.

The minimal length of a word labelling $\circled{1}$ is exactly $d_2(\a_2, \b_2)$. To see this let $w_2 = x_1^{\epsilon_1}\dots x_k^{\epsilon_k}$ with $\epsilon_i\in \left\{\pm 1\right\}$ be a word of minimal length representing $\a_2^{-1}\b_2$ in the generators $Y_2\cup Z_2$ of $G_i$ and their inverses. For $1\leq i\leq k$ we either have that $(s_1(\phi_2(x_i))^{-1}, x_i, e)$ is in $T_1$, if $x_i\in Z_2$, or in $U_2$, if $x_i\in Y_2$. The word $(s_1(\phi_2(x_1))^{-1}, x_1, e)\dots (s_1(\phi_2(x_k))^{-1}, x_k, e)$ has length $d_2(\a_2, \b_2)$ and the corresponding path connects the endpoints of $\circled{1}$. This provides the desired upper bound. For the lower bound observe that we can project any path in the level set to the second factor to get a path in $G_2$ from $\a_2$ to $\b_2$ in the generators $Y_2\cup Z_2$.

The edge $\circled{2}$ is labelled by a geodesic word in $U_3\cup T$. 
Using $d_3(\a_3, \b_3)$ generators we can obtain a vertex of the form $(\a_1', \overline{\b}_2, \b_3)$ for some $\overline{\b}_2\in G_2$. 
Since $\a_1\cdot (\a_1')^{-1}\in s_1(\ZZ^m)$, we can now apply $d_{\ZZ^m}(\phi_1(\a_1),\phi_1(\a_1'))=d_1(\a_1, \a_1')$ generators from $T_1$ to obtain the vertex $(\a_1, \b_2'', \b_3)$.  
Since $\a_1'$ was obtained from $\a_1$ via the path labelling the edge $\circled{1}$, we see that $d_{\ZZ^m}(\phi_1(\a_1),\phi_1(\a_1'))=d_{\ZZ^m}(\phi_2(\a_2),\phi_2(\b_2))\leq d_2(\a_2, \b_2)$. 
Thus the length of the path labelling the edge $\circled{2}$ is bounded above by $d_3(\a_3, \b_3) + d_2(\a_2, \b_2)$. 

The edge $\circled{3}$ is labelled by a geodesic word in the elements of $U_3\cup T_3$ and thus of length $d_3(\a_3, \c_3)$. 

Finally, the edge $\circled{4}$ is labelled by a geodesic path in the elements of $T$. Its length is bounded by the number of letters from $T$ in the words labelling the other three edges of the bounded region with label $\langle U_3\cup T\rangle$. 
Thus, it is bounded above by $d_3(\a_3, \b_3) + d_2(\a_2, \b_2) + d_3(\a_3, \c_3) + d_3(\b_3, \c_3)$. 

Let $r_{\a \b } = \max_i\{d_i(\a_i, \b_i)\}$, $r_{\b \c} = \max_i\{d_i(\b_i, \c_i)\}$ and $r_{\a \c} = \max_i\{d_i(\a_i, \c_i)\}$. 
We denote $U:=r_{\a \b} + r_{\b \c} + r_{\a \c}$ and observe that $U$ is bounded above by $D = d(\a, \b) + d(\a, \c) + d(\b, \c)$, where $d$ denotes the product metric on $G_1\times G_2\times G_3$.
We will refer to $D$ as the {\em perimeter} of the triangle. 
This will permit us to bound the area of our triangle purely in terms of $D$. 
Indeed, all edges of the triangle have length bounded by $3U$. 
Thus, the boundary of each region has length at most $12U$. 
In particular, we can obtain a filling of area bounded by $f(12U)$ for each of the null-homotopic words labelling the bounded regions. 

We deduce that the area of the triangle in $K$ is bounded above by $7f(12U)\leq 7f(12D)$. 

\begin{figure}
	\centering
	\begin{tikzpicture}[scale = 0.15]
	\pgfmathsetmacro{\rtt}{1.73}
	\pgfmathsetmacro{\wid}{51}
	\pgfmathsetmacro{\ab}{1}

	\node [label={[label distance=-0.1cm]below:\scriptsize $(\c_1, \c_2, \c_3)$}] (c) at (0, 0) {};
	\node [label={[label distance=-0.1cm]below:\scriptsize $(\b_1, \b_2, \b_3)$}] (b) at (\wid, 0) {};
	\node [label={[label distance=-0.1cm]above:\scriptsize $(\a_1, \a_2, \a_3)$}] (a) at (\wid/2, \wid*\rtt/2) {};
	
	\node [label={[label distance=-0.051cm]right:\scriptsize $(\a_1''', \b_2, \c_3)$}] (ma) at (\wid/2, \wid*\rtt/6 + \wid/\rtt/3) {};
	\node [label={[label distance=-0.25cm]below left:\scriptsize $(\a_1, \b_2, \c_3''')$}] (mc) at (\wid/3, \wid/\rtt/3) {};
	\node [label={[label distance=-0.25cm]below right:\scriptsize $(\a_1, \b_2''', \c_3)$}] (mb) at (2*\wid/3, \wid/\rtt/3) {};

	\node [label={[label distance=-0.19cm]above right:\scriptsize $(\a_1', \b_2, \a_3)$}] (aab) at (\wid/2 + \wid/6, \wid*\rtt/6 + \wid/\rtt/3 + \wid/6/\rtt) {};
	\node [label={[label distance=-0.19cm]above right:\scriptsize $(\a_1, \b_2'', \b_3)$}] (abb) at (2*\wid/3 + \wid/6, \wid/\rtt/3 + \wid/6/\rtt) {};
	
	\node [label={[label distance=-0.1cm]below:\scriptsize $(\b_1, \b_2', \c_3)$}] (bbc) at (2*\wid/3, 0) {};
	\node [label={[label distance=-0.1cm]below:\scriptsize $(\c_1, \b_2, \c_3'')$}] (bcc) at (\wid/3, 0) {};
	
	\node [label={[label distance=-0.19cm]above left:\scriptsize $(\a_1'', \a_2, \c_3)$}] (aac) at (\wid/2 -\wid/6, \wid*\rtt/6 + \wid/\rtt/3 + \wid/6/\rtt) {};
	\node [label={[label distance=-0.19cm]above left:\scriptsize $(\a_1, \c_2, \c_3')$}] (acc) at (\wid/3 - \wid/6, \wid/\rtt/3 + \wid/6/\rtt) {};

	\fill (a) circle (2em);
	\fill (b) circle (2em);
	\fill (c) circle (2em);
	\fill (ma) circle (2em);
	\fill (mb) circle (2em);
	\fill (mc) circle (2em);
	\fill (aab) circle (2em);
	\fill (abb) circle (2em);
	\fill (aac) circle (2em);
	\fill (acc) circle (2em);
	\fill (bcc) circle (2em);
	\fill (bbc) circle (2em);
	
	\draw (0, 0) -- (\wid, 0) -- (\wid/2, \wid*\rtt/2) -- (0,0);
	\draw (\wid/3, \wid/\rtt/3) -- (2*\wid/3, \wid/\rtt/3) -- (\wid/2, \wid*\rtt/6 + \wid/\rtt/3) --(\wid/3, \wid/\rtt/3);

	\draw (\wid/2 + \wid/6, \wid*\rtt/6 + \wid/\rtt/3 + \wid/6/\rtt) -- (\wid/2, \wid*\rtt/6 + \wid/\rtt/3);
	\draw (2*\wid/3 + \wid/6, \wid/\rtt/3 + \wid/6/\rtt) -- (2*\wid/3, \wid/\rtt/3);
	
	\draw (2*\wid/3, 0) -- (2*\wid/3, \wid/\rtt/3);
	\draw (\wid/3, 0) -- (\wid/3, \wid/\rtt/3);
	
	\draw (\wid/2 -\wid/6, \wid*\rtt/6 + \wid/\rtt/3 + \wid/6/\rtt) -- (\wid/2, \wid*\rtt/6 + \wid/\rtt/3);
	\draw (\wid/3 - \wid/6, \wid/\rtt/3 + \wid/6/\rtt) -- (\wid/3, \wid/\rtt/3);

	\node () at (18*\wid/36, 18*\wid*\rtt/49 +1) {$\scriptscriptstyle\langle U_2\cup U_3\cup T \rangle$};
	\node [rotate = -60] () at (24.5*\wid/36 +0.9, 44*\wid*\rtt/200 +0.3) {$\scriptscriptstyle\langle U_3\cup T \rangle$};
	\node [rotate = 60]  () at (8.5*\wid/26-0.9, 44*\wid*\rtt/200 +0.3) {$\scriptscriptstyle\langle U_2\cup T\rangle$};
	\node () at (18*\wid/36, 34*\wid*\rtt/200+1) {$\scriptscriptstyle\langle T \rangle$};
	\node () at (29*\wid/36, 9*\wid*\rtt/200+1) {$\scriptscriptstyle\langle U_1\cup U_3\cup T\rangle$};
	\node () at (18*\wid/36, 9*\wid*\rtt/200+1) {$\scriptscriptstyle\langle U_1\cup T \rangle$};
	\node () at (7*\wid/36, 9*\wid*\rtt/200+1) {$\scriptscriptstyle\langle U_1\cup U_2\cup T \rangle$};

	\node () at (18*\wid/36, 18*\wid*\rtt/49-1) {$\scriptscriptstyle \cong G_2\times G_3$};
	\node [rotate = -60] () at (24.5*\wid/36 -0.9, 44*\wid*\rtt/200 -0.3) {$\scriptscriptstyle \cong G_3\times \ZZ^m$};
	\node [rotate = 60]  () at (8.5*\wid/26 +0.9, 44*\wid*\rtt/200 -0.3) {$\scriptscriptstyle \cong G_2\times \ZZ^m$};
	\node () at (18*\wid/36, 34*\wid*\rtt/200-1) {$\scriptscriptstyle \cong \ZZ^{2m}$};
	\node () at (29*\wid/36, 9*\wid*\rtt/200-1) {$\scriptscriptstyle \cong G_1\times G_3$};
	\node () at (18*\wid/36, 9*\wid*\rtt/200-1) {$\scriptscriptstyle \cong G_1\times \ZZ^m$};
	\node () at (7*\wid/36, 9*\wid*\rtt/200-1) {$\scriptscriptstyle \cong G_1\times G_2$};

	\node () at (22.5*\wid/36, 20.5*\wid*\rtt/49) {\circled{1}};
	\node () at (28.5*\wid/36, 12.5*\wid*\rtt/49) {\circled{2}};
	\node () at (20.5*\wid/36, 16*\wid*\rtt/49) {\circled{3}};
	\node () at (22*\wid/36, 10*\wid*\rtt/49) {\circled{4}};

\end{tikzpicture}	
	\caption{An algebraic spanning triangle.}
	\label{fig:algebraicspanningtriangle}
\end{figure}
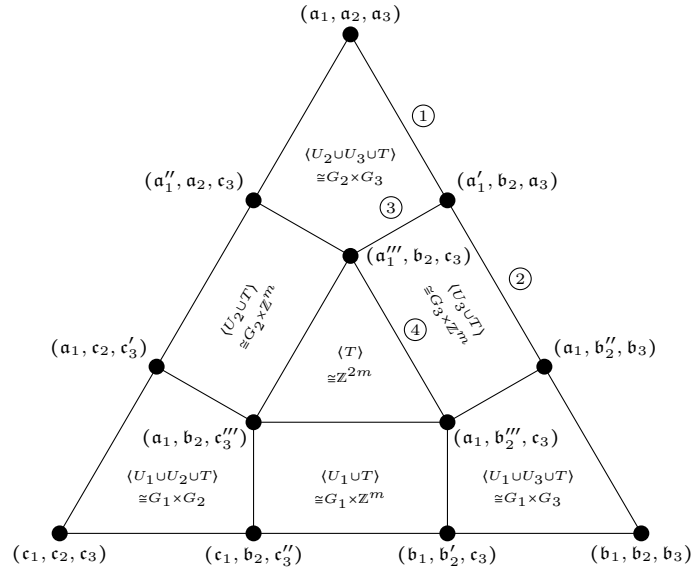

\begin{figure}
	\centering
	\begin{tikzpicture}[scale = 0.15]
	\pgfmathsetmacro{\rtt}{1.73}
	\pgfmathsetmacro{\wid}{70}
	\pgfmathsetmacro{\ab}{1}

	\node [label={[label distance=-0.1cm]below:\scriptsize $(\c_1, \c_2, \c_3)$}] (c) at (0, 0) {};
	\node [label={[label distance=-0.1cm]below:\scriptsize $(\b_1, \b_2, \b_3)$}] (b) at (\wid, 0) {};
	\node [label={[label distance=-0.1cm]above:\scriptsize $(\a_1, \a_2, \a_3)$}] (a) at (\wid/2, \wid*\rtt/2) {};
	
	\node   (ma) at (\wid/2, \wid*\rtt/6 + \wid/\rtt/3) {};
	\node   (mc) at (\wid/3, \wid/\rtt/3) {};
	\node   (mb) at (2*\wid/3, \wid/\rtt/3) {};

	\node   (aab) at (\wid/2 + \wid/6, \wid*\rtt/6 + \wid/\rtt/3 + \wid/6/\rtt) {};
	\node   (abb) at (2*\wid/3 + \wid/6, \wid/\rtt/3 + \wid/6/\rtt) {};
	
	\node   (bbc) at (2*\wid/3, 0) {};
	\node   (bcc) at (\wid/3, 0) {};
	
	\node   (aac) at (\wid/2 -\wid/6, \wid*\rtt/6 + \wid/\rtt/3 + \wid/6/\rtt) {};
	\node  (acc) at (\wid/3 - \wid/6, \wid/\rtt/3 + \wid/6/\rtt) {};

	\fill (a) circle (2em);
	\fill (b) circle (2em);
	\fill (c) circle (2em);
	\fill (ma) circle (2em);
	\fill (mb) circle (2em);
	\fill (mc) circle (2em);
	\fill (aab) circle (2em);
	\fill (abb) circle (2em);
	\fill (aac) circle (2em);
	\fill (acc) circle (2em);
	\fill (bcc) circle (2em);
	\fill (bbc) circle (2em);
	
	\draw (0, 0) -- (\wid, 0) -- (\wid/2, \wid*\rtt/2) -- (0,0);
	\draw (\wid/3, \wid/\rtt/3) -- (2*\wid/3, \wid/\rtt/3) -- (\wid/2, \wid*\rtt/6 + \wid/\rtt/3) --(\wid/3, \wid/\rtt/3);

	\draw (\wid/2 + \wid/6, \wid*\rtt/6 + \wid/\rtt/3 + \wid/6/\rtt) -- (\wid/2, \wid*\rtt/6 + \wid/\rtt/3);
	\draw (2*\wid/3 + \wid/6, \wid/\rtt/3 + \wid/6/\rtt) -- (2*\wid/3, \wid/\rtt/3);
	
	\draw (2*\wid/3, 0) -- (2*\wid/3, \wid/\rtt/3);
	\draw (\wid/3, 0) -- (\wid/3, \wid/\rtt/3);
	
	\draw (\wid/2 -\wid/6, \wid*\rtt/6 + \wid/\rtt/3 + \wid/6/\rtt) -- (\wid/2, \wid*\rtt/6 + \wid/\rtt/3);
	\draw (\wid/3 - \wid/6, \wid/\rtt/3 + \wid/6/\rtt) -- (\wid/3, \wid/\rtt/3);

	\node [rotate = -60] () at (22.5*\wid/36, 20.5*\wid*\rtt/49) {\scriptsize $d_2(\a_2, \b_2)$};
	\node [rotate = -60] () at (34*\wid/36, 4.5*\wid*\rtt/49)    {\scriptsize $d_1(\a_1, \b_1)$};
	\node () at (30*\wid/36, -1*\wid*\rtt/49)     {\scriptsize $d_3(\b_3, \c_2)$};
	\node () at (6*\wid/36, -1*\wid*\rtt/49)      {\scriptsize $d_2(\b_2, \c_2)$};
	\node [rotate = 60] () at (2*\wid/36, 4.5*\wid*\rtt/49)     {\scriptsize  $d_1(\a_1, \c_1)$};
	\node [rotate = 60] () at (13.5*\wid/36, 20.5*\wid*\rtt/49) {\scriptsize $d_3(\a_3, \c_3)$};
	
	\node [rotate = -60] () at (28.5*\wid/36, 12.5*\wid*\rtt/49) {\scriptsize $d_2(\a_2, \b_2) + d_3(\a_3, \b_3)$};
	\node () at (18*\wid/36, -1*\wid*\rtt/49)     {\scriptsize $d_1(\b_1, \c_1) + d_3(\b_3, \c_3)$};
	\node [rotate = 60] () at (7.5*\wid/36, 12.5*\wid*\rtt/49)  {\scriptsize $d_1(\a_1, \c_1) + d_2(\a_2, \c_2)$};
	
	\node[rotate = 30] () at (20.5*\wid/36, 16*\wid*\rtt/49)   {\scriptsize $d_3(\a_3, \c_3)$};
	\node[rotate = 30] () at (27.5*\wid/36, 6*\wid*\rtt/49)    {\scriptsize $d_3(\b_3, \c_3)$};
	\node[rotate = -90] () at (25*\wid/36, 2.3*\wid*\rtt/49)  {\scriptsize $d_1(\a_1, \b_1)$};
	\node[rotate = 90] () at (11*\wid/36, 2.3*\wid*\rtt/49)  {\scriptsize $d_1(\a_1, \c_1)$};
	\node[rotate = -30] () at (8.5*\wid/36, 6*\wid*\rtt/49)     {\scriptsize $d_2(\c_2, \b_2)$};
	\node[rotate = -30] () at (15.5*\wid/36, 16*\wid*\rtt/49)   {\scriptsize $d_2(\a_2, \b_2)$};
	
	\node[rotate = -60] () at (22.5*\wid/36, 10*\wid*\rtt/49)     {\scriptsize $\begin{aligned} d_2(\a_2, \b_2) &+ d_3(\a_3, \b_3) +\\d_3(\a_3, \c_3) &+ d_3(\b_3, \c_3)\end{aligned}$};
	\node () at (18*\wid/36, 4*\wid*\rtt/49)    {\scriptsize $\begin{aligned}d_1(\a_1, \b_1) &+ d_1(\a_1, \c_1) + \\d_1(\b_1, \c_1) &+ d_3(\b_3, \c_3)\end{aligned}$};
	\node [rotate = 60] () at (13.5*\wid/36, 10*\wid*\rtt/49)     {\scriptsize $\begin{aligned}d_1(\a_1, \c_1) &+ d_2(\a_2, \b_2) + \\d_2(\a_2, \c_2) &+ d_2(\b_2, \c_2)\end{aligned}$};

\end{tikzpicture}
	\caption{Length bounds on the words labelling the edges of the algebraic spanning triangle in \Cref{fig:algebraicspanningtriangle}.}
	\label{fig:lengthtri}
\end{figure}
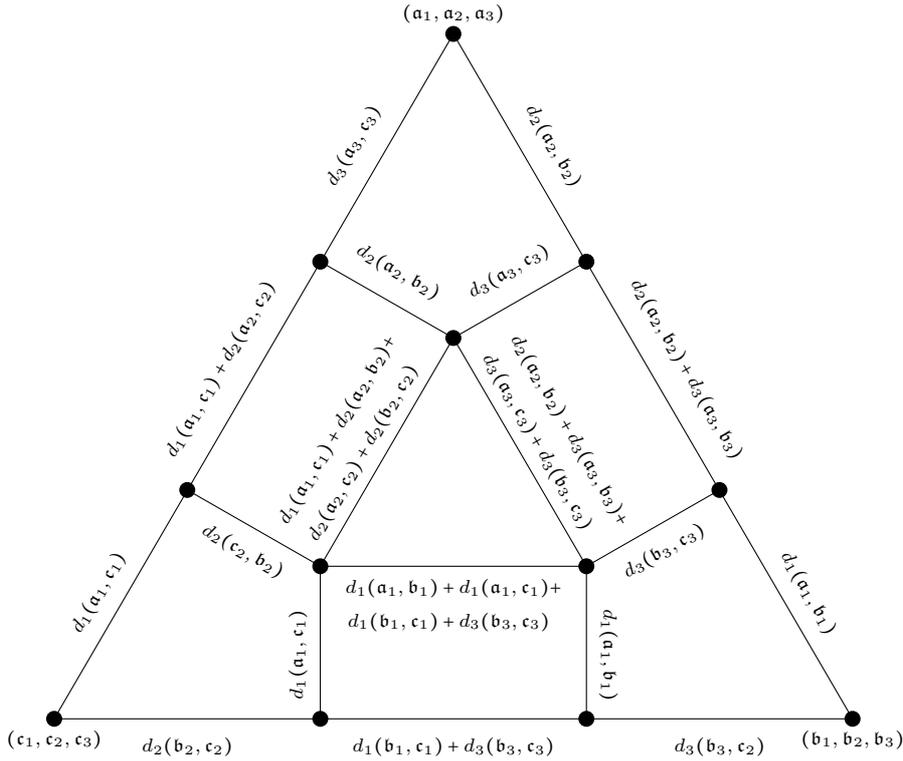

\begin{remark}\label{rem:bigonshavezeroarea}
	Suppose that $\a$ and $\b$ differ by a single generator of $U\cup T$. 
	Then the path between them in \Cref{fig:algebraicspanningtriangle} has length at most 2. 
	To see this, note that the length of the path is bounded by $d_1(\a_1, \b_1) + 2\cdot d_2(\a_2, \b_2)+ d_3(\a_3, \b_3)$. 
	At least two of these quantities vanish. 
	Thus we obtain the desired bound. 
\end{remark}

We can now turn to the remainder of the proof of Theorem \ref{thm:algebraictriangle}.
The rest of the argument proceeds similarly to that of \cite{CarFor-17}.

\begin{proof}[Proof of \cref{thm:algebraictriangle}]
	Let $w$ be a null-homotopic word in $K=\langle U\cup T\rangle$ of length $n\geq 3$ and let $k\in\ZZ$ be such that $3\cdot2^{k-1}\leq n\leq 3\cdot 2^{k}$. Let $\widehat{\gamma}:\left[0,n\right]\to Cay(K,U\cup T)$ be the loop corresponding to $w$ in the Cayley graph of $K$ parametrised by length and based at the identity.
	We can extend $\widehat{\gamma}$ to a loop $\gamma:\left[0,3\cdot 2^k\right]\to Cay(K,U\cup T)$ by adding a constant path to the end. Note that $\widehat{\gamma}$ maps all integers in $\left[n,3\cdot 2^k\right]$ to the vertex corresponding to the trivial element.
	Since we have appended a trivial path, $\gamma$ and $\hat{\gamma}$ have the same area. 
	We will show that $\gamma$ has area $\preccurlyeq f(n)$.

	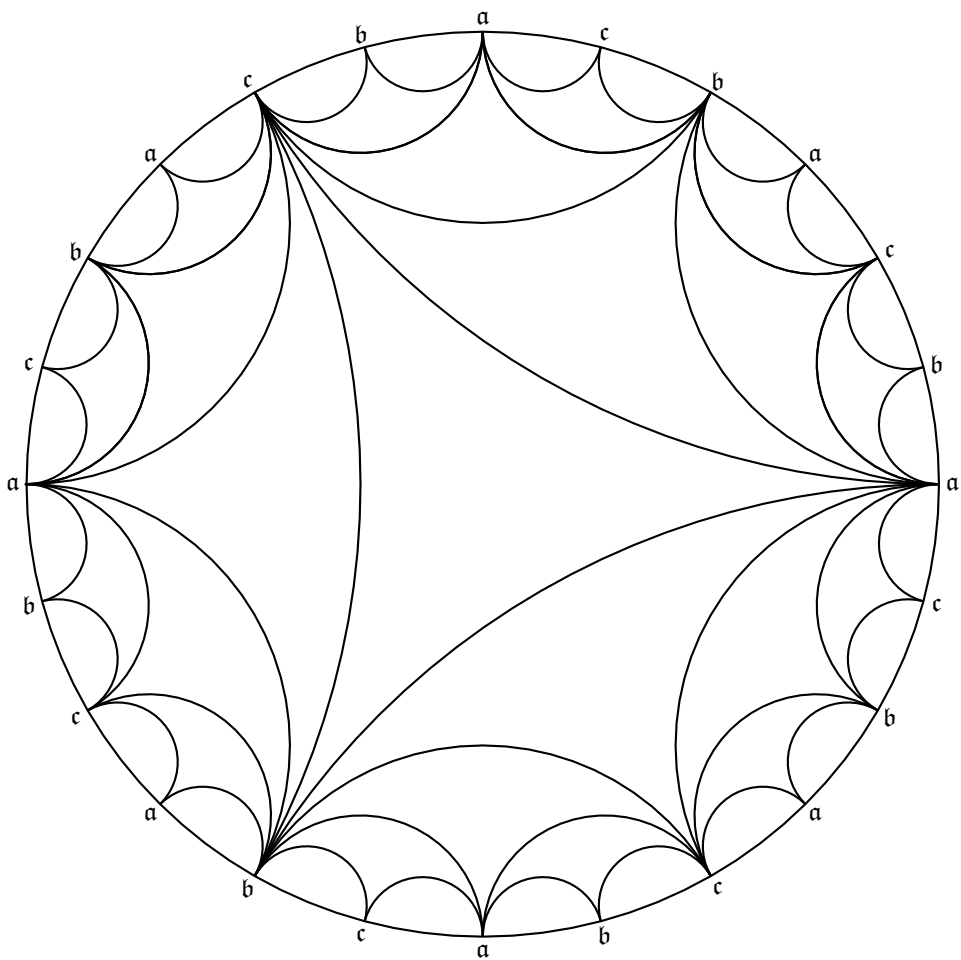
\begin{figure}
			\begin{tikzpicture}[scale = 6]
				\draw[black, thick, domain = 0:90] plot ({cos(\x)}, {sin(\x)});
				\draw[black, thick, domain = 180:270] plot ({cos(\x)}, {sin(\x)});
				\draw[black, thick, domain = 270:360] plot ({cos(\x)}, {sin(\x)});
				\draw[black, thick, domain = 90:180] plot ({cos(\x)}, {sin(\x)});
				
				\draw [black,thick,domain=269.613242082:210.00088480270372] plot ({1.0+1.73205080757*cos(\x)}, {1.7320508075688767+1.73205080757*sin(\x)});
				\draw [black,thick,domain=30.000884802703716:-30.00088480270375] plot ({-1.9999999999999996+1.73205080757*cos(\x)}, {8.881784197001262e-16+1.73205080757*sin(\x)});
				\draw [black,thick,domain=149.99911519729622:90.3867579182] plot ({1.0+1.73205080757*cos(\x)}, {-1.7320508075688785+1.73205080757*sin(\x)});
				\draw [black,thick,domain=269.613242082:149.99911519729625] plot ({1.0+0.57735026919*cos(\x)}, {0.5773502691896257+0.57735026919*sin(\x)});
				\draw [black,thick,domain=329.99911519729625:210.00088480270372] plot ({3.8459253727671266e-16+0.57735026919*cos(\x)}, {1.1547005383792515+0.57735026919*sin(\x)});
				\draw [black,thick,domain=30.000884802703716:-90.3867579182] plot ({-1.0+0.57735026919*cos(\x)}, {0.5773502691896261+0.57735026919*sin(\x)});
				\draw [black,thick,domain=90.3867579182:-30.00088480270375] plot ({-1.0+0.57735026919*cos(\x)}, {-0.5773502691896254+0.57735026919*sin(\x)});
				\draw [black,thick,domain=149.99911519729625:30.000884802703737] plot ({-3.8459253727671227e-16+0.57735026919*cos(\x)}, {-1.1547005383792517+0.57735026919*sin(\x)});
				\draw [black,thick,domain=210.00088480270375:90.3867579182] plot ({1.0+0.57735026919*cos(\x)}, {-0.5773502691896256+0.57735026919*sin(\x)});
				\draw [black,thick,domain=269.613242082:119.9982303945925] plot ({1.0+0.267949192431*cos(\x)}, {0.26794919243112264+0.267949192431*sin(\x)});
				\draw [black,thick,domain=299.9982303945925:149.99911519729628] plot ({0.7320508075688774+0.267949192431*cos(\x)}, {0.7320508075688773+0.267949192431*sin(\x)});
				\draw [black,thick,domain=329.99911519729625:180.0] plot ({0.26794919243112286+0.267949192431*cos(\x)}, {1.0+0.267949192431*sin(\x)});
				\draw [black,thick,domain=359.99999999999994:210.00088480270372] plot ({-0.26794919243112236+0.267949192431*cos(\x)}, {1.0000000000000002+0.267949192431*sin(\x)});
				\draw [black,thick,domain=30.000884802703702:-119.99823039459253] plot ({-0.7320508075688771+0.267949192431*cos(\x)}, {0.7320508075688776+0.267949192431*sin(\x)});
				\draw [black,thick,domain=60.001769605407496:-90.3867579182] plot ({-1.0+0.267949192431*cos(\x)}, {0.26794919243112264+0.267949192431*sin(\x)});
				\draw [black,thick,domain=89.99734559188883:-60.001769605407446] plot ({-0.9999999999999999+0.267949192431*cos(\x)}, {-0.2679491924311232+0.267949192431*sin(\x)});
				\draw [black,thick,domain=119.99823039459247:-30.00088480270375] plot ({-0.732050807568878+0.267949192431*cos(\x)}, {-0.7320508075688767+0.267949192431*sin(\x)});
				\draw [black,thick,domain=149.99911519729622:0.0] plot ({-0.26794919243112303+0.267949192431*cos(\x)}, {-1.0+0.267949192431*sin(\x)});
				\draw [black,thick,domain=179.99999999999994:30.00088480270373] plot ({0.2679491924311222+0.267949192431*cos(\x)}, {-1.0000000000000002+0.267949192431*sin(\x)});
				\draw [black,thick,domain=210.00088480270378:60.00176960540744] plot ({0.7320508075688772+0.267949192431*cos(\x)}, {-0.7320508075688771+0.267949192431*sin(\x)});
				\draw [black,thick,domain=240.00176960540745:90.3867579182] plot ({1.0+0.267949192431*cos(\x)}, {-0.26794919243112303+0.267949192431*sin(\x)});
				\draw [black,thick,domain=269.613242082:119.99823039459254] plot ({1.0+0.267949192431*cos(\x)}, {0.26794919243112286+0.267949192431*sin(\x)});
				\draw [black,thick,domain=299.9982303945925:149.9991151972962] plot ({0.7320508075688775+0.267949192431*cos(\x)}, {0.7320508075688769+0.267949192431*sin(\x)});
				\draw [black,thick,domain=329.99911519729625:179.99999999999991] plot ({0.26794919243112353+0.267949192431*cos(\x)}, {0.9999999999999997+0.267949192431*sin(\x)});
				\draw [black,thick,domain=359.9999999999999:210.0008848027037] plot ({-0.26794919243112164+0.267949192431*cos(\x)}, {1.0000000000000004+0.267949192431*sin(\x)});
				\draw [black,thick,domain=30.000884802703656:-119.99823039459247] plot ({-0.7320508075688775+0.267949192431*cos(\x)}, {0.7320508075688776+0.267949192431*sin(\x)});
				\draw [black,thick,domain=60.00176960540746:-89.9973455918888] plot ({-0.9999999999999999+0.267949192431*cos(\x)}, {0.2679491924311232+0.267949192431*sin(\x)});
				\draw [black,thick,domain=269.613242082:104.99778799324064] plot ({1.0+0.131652497587*cos(\x)}, {0.13165249758739564+0.131652497587*sin(\x)});
				\draw [black,thick,domain=284.99778799324065:119.99823039459253] plot ({0.9318516525781365+0.131652497587*cos(\x)}, {0.3859855926176458+0.131652497587*sin(\x)});
				\draw [black,thick,domain=299.9982303945925:134.99867279594446] plot ({0.8001991549907406+0.131652497587*cos(\x)}, {0.6140144073823547+0.131652497587*sin(\x)});
				\draw [black,thick,domain=314.99867279594434:149.99911519729625] plot ({0.6140144073823548+0.131652497587*cos(\x)}, {0.8001991549907405+0.131652497587*sin(\x)});
				\draw [black,thick,domain=329.99911519729625:164.99955759864807] plot ({0.3859855926176463+0.131652497587*cos(\x)}, {0.9318516525781363+0.131652497587*sin(\x)});
				\draw [black,thick,domain=344.99955759864815:180.0] plot ({0.13165249758739553+0.131652497587*cos(\x)}, {1.0+0.131652497587*sin(\x)});
				\draw [black,thick,domain=0.0:-164.99955759864815] plot ({-0.13165249758739558+0.131652497587*cos(\x)}, {1.0+0.131652497587*sin(\x)});
				\draw [black,thick,domain=15.000442401351913:-149.99911519729625] plot ({-0.3859855926176463+0.131652497587*cos(\x)}, {0.9318516525781363+0.131652497587*sin(\x)});
				\draw [black,thick,domain=30.000884802703688:-134.9986727959444] plot ({-0.6140144073823539+0.131652497587*cos(\x)}, {0.8001991549907411+0.131652497587*sin(\x)});
				\draw [black,thick,domain=45.001327204055606:-119.99823039459255] plot ({-0.8001991549907407+0.131652497587*cos(\x)}, {0.6140144073823544+0.131652497587*sin(\x)});
				\draw [black,thick,domain=60.001769605407496:-104.99778799324059] plot ({-0.9318516525781368+0.131652497587*cos(\x)}, {0.3859855926176452+0.131652497587*sin(\x)});
				\draw [black,thick,domain=75.00221200675921:-90.00265440811114] plot ({-1.0000000000000002+0.131652497587*cos(\x)}, {0.13165249758739622+0.131652497587*sin(\x)});
				\draw [black,thick,domain=90.3867579182:-75.00221200675935] plot ({-1.0+0.131652497587*cos(\x)}, {-0.13165249758739564+0.131652497587*sin(\x)});
				\draw [black,thick,domain=104.99778799324065:-60.00176960540733] plot ({-0.9318516525781363+0.131652497587*cos(\x)}, {-0.3859855926176468+0.131652497587*sin(\x)});
				\draw [black,thick,domain=119.99823039459241:-45.001327204055656] plot ({-0.8001991549907419+0.131652497587*cos(\x)}, {-0.6140144073823529+0.131652497587*sin(\x)});
				\draw [black,thick,domain=134.99867279594434:-30.00088480270375] plot ({-0.6140144073823547+0.131652497587*cos(\x)}, {-0.8001991549907405+0.131652497587*sin(\x)});
				\draw [black,thick,domain=149.99911519729622:-15.000442401351904] plot ({-0.38598559261764614+0.131652497587*cos(\x)}, {-0.9318516525781364+0.131652497587*sin(\x)});
				\draw [black,thick,domain=164.99955759864815:0.0] plot ({-0.13165249758739558+0.131652497587*cos(\x)}, {-1.0+0.131652497587*sin(\x)});
				\draw [black,thick,domain=180.0:15.000442401351851] plot ({0.13165249758739575+0.131652497587*cos(\x)}, {-1.0+0.131652497587*sin(\x)});
				\draw [black,thick,domain=195.00044240135182:30.00088480270372] plot ({0.38598559261764503+0.131652497587*cos(\x)}, {-0.9318516525781368+0.131652497587*sin(\x)});
				\draw [black,thick,domain=210.00088480270378:45.00132720405557] plot ({0.6140144073823547+0.131652497587*cos(\x)}, {-0.8001991549907403+0.131652497587*sin(\x)});
				\draw [black,thick,domain=225.0013272040556:60.001769605407475] plot ({0.8001991549907406+0.131652497587*cos(\x)}, {-0.6140144073823546+0.131652497587*sin(\x)});
				\draw [black,thick,domain=240.00176960540742:75.00221200675936] plot ({0.9318516525781365+0.131652497587*cos(\x)}, {-0.38598559261764603+0.131652497587*sin(\x)});
				\draw [black,thick,domain=255.0022120067593:90.00265440811116] plot ({0.9999999999999999+0.131652497587*cos(\x)}, {-0.1316524975873968+0.131652497587*sin(\x)});
				\node () at (1.03, 0.0) {$\a$};
				\node () at (0.994903601078, 0.266583616456) {$\b$};
				\node () at (0.892006165898, 0.515) {$\c$};
				\node () at (0.728319984622, 0.728319984622) {$\a$};
				\node () at (0.515, 0.892006165898) {$\b$};
				\node () at (0.266583616456, 0.994903601078) {$\c$};
				\node () at (6.30693101561e-17, 1.03) {$\a$};
				\node () at (-0.266583616456, 0.994903601078) {$\b$};
				\node () at (-0.515, 0.892006165898) {$\c$};
				\node () at (-0.728319984622, 0.728319984622) {$\a$};
				\node () at (-0.892006165898, 0.515) {$\b$};
				\node () at (-0.994903601078, 0.266583616456) {$\c$};
				\node () at (-1.03, 1.26138620312e-16) {$\a$};
				\node () at (-0.994903601078, -0.266583616456) {$\b$};
				\node () at (-0.892006165898, -0.515) {$\c$};
				\node () at (-0.728319984622, -0.728319984622) {$\a$};
				\node () at (-0.515, -0.892006165898) {$\b$};
				\node () at (-0.266583616456, -0.994903601078) {$\c$};
				\node () at (-1.89207930468e-16, -1.03) {$\a$};
				\node () at (0.266583616456, -0.994903601078) {$\b$};
				\node () at (0.515, -0.892006165898) {$\c$};
				\node () at (0.728319984622, -0.728319984622) {$\a$};
				\node () at (0.892006165898, -0.515) {$\b$};
				\node () at (0.994903601078, -0.266583616456) {$\c$};
			\end{tikzpicture}
		\caption{A loop subdivided into bigons and triangles}
		\label{fig:hyptriangle}
	\end{figure}

	Let $D$ be the disk shown in \cref{fig:hyptriangle}. 
	It has $l$ vertices on the boundary, $l$ bigons adjacent to the boundary and $3\cdot 2^k - 2$ triangles. 
	Each triangle has a {\em depth}, the central triangle is at depth 0 and its neighbours are at depth 1 and so on. 
	For $k\geq i\geq 1$ there are $3\cdot 2^{i-1}$ triangles of depth $i$.

	There is a labelling of the vertices of \cref{fig:hyptriangle} by the set $\{\a, \b, \c\}$ such that the boundary of each triangle reads $\a, \b, \c$ or $\c, \b, \a$. 
	To obtain such a labelling we proceed as follows. 
	Each edge is in exactly two triangles, pick a labelling on the central triangle such that the boundary reads $\a,\b,\c$.
	Label the vertices of all triangles of level $i$ by reflecting the labelling of a triangle of level $i-1$ along the edge joining them. 
	
	Each triangle can now be filled with the spanning triangle with the given three boundary points. 
	The reflection technique used for the labelling allows us to choose the paths in the spanning triangles so that they agree on edges in their intersection. 
	
	Give each edge a {\em depth} by declaring it to be the minimum depth of triangles adjacent to the edge.
	Observe that the boundary vertices of an edge of depth $i$ are at distance $\leq 2^{k-i}$.
	
	The perimeter of the central triangle is bounded by $3\cdot 2^k$, and,  
	for each $k\geq i\geq 1$, the perimeter of a triangle of depth $i$ is bounded by $2^{k - i +1} + 2\cdot 2^{k-i} = 2^{k-i+2}$. 	
	Thus it follows from the previous section that the central spanning triangle has area $\leq 7f(12\cdot 3\cdot 2^k)$ and each spanning triangle of depth $i$ has area $\leq 7f(12\cdot 2^{k-i+2})$.
	
	Also, by \cref{rem:bigonshavezeroarea}, we see that each bigon has perimeter $\leq 4$ implying that there is a uniform bound $B>0$ on the area of all bigons appearing in our fillings. 
	
	We deduce that the area enclosed by $\gamma$ is
	$$\leq 7f(12\cdot 3\cdot 2^k) + \sum_{i=1}^k 7\cdot 3\cdot 2^{i-1}\cdot f(12\cdot 2^{k-i+2}) + 3\cdot 2^k\cdot B.$$ 

	Define $f'\colon \NN\to \RR_{>0}$ by $f'(n) := f(n)/n$ and assume that $f'$ is superadditive. 
	We obtain the following estimates:  
	\begin{align*}
	\sum_{i=1}^k 7\cdot 3\cdot 2^{i-1} f(12\cdot 2^{k-i+2})
	& = \sum_{i=1}^k 7\cdot 3\cdot 2^{i-1} \cdot 12\cdot 2^{k-i+2} f'(12\cdot 2^{k-i+2})\\
	&\leq 7\cdot 3\cdot 12 \cdot 2^{k+1} \sum_{i=1}^k f'(12\cdot 2^{k-i+2})\\
	&\leq 7\cdot 3\cdot 12 \cdot 2^{k+1} f'(\sum_{i=1}^k 12\cdot 2^{k-i+2})\\
	&\leq 7\cdot 3\cdot 12 \cdot 2^{k+2} f'(12\cdot 2^{k+2})\\
	& = 7\cdot 3\cdot f(12\cdot 2^{k+2})\\
	\end{align*}
	Thus we see that $\gamma$ has area bounded by 
	\begin{align*}
	7f(12\cdot 3\cdot 2^k) + 7\cdot 3\cdot f(12\cdot 4 \cdot2^k) + 3\cdot 2^k\cdot B
	&\leq 7\cdot 4 \cdot f(12\cdot 4\cdot 2^k) + 3\cdot 2^k\cdot B\\
	&\leq 28\cdot f(12\cdot 4 \cdot n) + 3\cdot n\cdot B\preccurlyeq f(n),
	\end{align*}
	where for the last line we use that $n\geq 2^k$. This provides us with the desired upper bound. 
	
	To obtain the lower bound, note that there are retractions $K\to G_i\times G_j$ for each choice of $i\neq j$. 
	Also by \cref{lem:Dehndirectproduct} we obtain that the maximum of the Dehn functions of the $G_i\times G_j$ is equivalent to that of $G_1\times G_2\times G_3$ and is also equivalent to $f$. 
	Thus we can conclude that the Dehn function of $K$ is equivalent to $f$

	In the case that $f(n)/n$ is not superadditive we proceed as follows. 
	Once again, we see that the area enclosed by $\gamma$ is
	$$\leq 7f(12\cdot 3\cdot 2^k) + \sum_{i=1}^k 7\cdot 3\cdot 2^{i-1}\cdot f(12\cdot 2^{k-i+2}) + 3\cdot 2^k\cdot B.$$ 
	
	Let $\bar{f}$ be the superadditive closure of $f$. We now get the following inequalities:
	\begin{align*}
	\sum_{i=1}^k 7\cdot 3\cdot 2^{i-1} f(12\cdot 2^{k-i+2})
	& \leq	\sum_{i=1}^k 7\cdot 3\cdot 2^{i-1} \bar{f}(12\cdot 2^{k-i+2})\\
	& \leq	\sum_{i=1}^k 7\cdot 3 \bar{f}(12\cdot 2^{k-1})\\
	& \leq	7\cdot k\cdot \bar{f}(3\cdot 12\cdot 2^{k})\\
	\end{align*}
	Thus we see that $\gamma$ has area bounded by 
	\begin{align*}
	7f(12\cdot 3\cdot 2^k) + 7\cdot k\cdot \bar{f}(3\cdot 12\cdot 2^{k}) + 3\cdot 2^k\cdot B
	&\leq 7(k+1)\bar{f}(3\cdot 12\cdot 2^{k}) + 3\cdot 2^k\cdot B\\
	&\leq 7(\log{n} + 1)\bar{f}(3\cdot 12\cdot n) + 3\cdot 2n\cdot B\\
	&\preccurlyeq \log{n}\cdot \bar{f}(n). 
	\end{align*}
\end{proof}

\begin{proof}[Proof of Corollary \ref{cor:triangle-thm-3}]
 For a null-homotopic word $w(X)$ of length $\leq n$ with corresponding edge loop $\gamma$ in $\mathrm{Cay}(K,U\cup T)$ our method provides us with a filling which is a product of conjugates of fillings for the boundary words of triangles and bigons as in Figure \ref{fig:hyptriangle} by subwords of $w(X)$. The same arguments as above show that we can construct fillings for the triangles such that the total area of the filling of $\gamma$ is $\preccurlyeq log(n)\cdot \overline{f}(n)$, where we may now choose the fillings of the regions of the triangle so that their filling diameter is $\leq g(12\cdot 3\cdot 2^k)\preccurlyeq g(n)$. Indeed, the length of the boundary loop of every region of a triangle is bounded above by the length of the boundary loop of the triangle of depth $0$, which is $\leq 12\cdot 3\cdot 2^k$, thus the filling diameter for every triangle is $\preccurlyeq g(n)$.
 
Since all triangles and bigons have a vertex on $\gamma$, the filling diameter of our filling for $w(X)$ is $\preccurlyeq n + g(n)\asymp g(n)$. We deduce that $(log(n)\cdot \overline{f}(n), g(n))$ is a filling pair for $K$.
\end{proof}

\section{Algebraic square method}
\label{sec:square}

To prove Theorem \ref{thm:algebraictriangle} we required that all three of the short exact sequences split. 
We now prove a theorem where we can relax the condition of splitting to the following:
\begin{definition}\label{def:Psplit}
	We say that $P = \{A, B\}$ is a {\em factoring} of $\ZZ^m$ if $A, B \leq  \ZZ^m$ and $\ZZ^m = A\oplus B$. 
	
	Let $1\to N\to G\xrightarrow{\phi} \ZZ^m\to 1$ be a short exact sequence. 
	Given a factoring $P$, we say that the short exact sequence {\em $P$-splits} if there are maps \begin{align*}
		s^1&\colon A\to G,\\
		s^2&\colon B\to G,
	\end{align*}
	such that $\phi\circ s^i = Id$. 
\end{definition}

It is clear that if a short exact sequence splits, then it $P$-splits for any $P$. 
The abelianization map $F_2\to \ZZ^2$ provides an example which does not split, but $P$-splits for certain $P$. Indeed it $P$-splits precisely if $A$ and $B$ are both non-trivial. 
With this terminology we can state the main result of this section. 

\algebraicsquare

In analogy to Corollaries \ref{cor:triangle-thm-1} and \ref{cor:triangle-thm-3} we deduce from Theorem \ref{thm:algebraicsquare} and its proof:
\begin{corollary}\label{cor:square-thm-1}
	If in Theorem \ref{thm:algebraicsquare} $\frac{f(n)}{n^2}$ is non-decreasing, then $\ker(\phi)$ is finitely presented and has Dehn function $f$.
\end{corollary}
\begin{corollary}\label{cor:square-thm-3}
If $(f_i,g_i)$ is a filling pair for $G_i$ then the proof of Theorem \ref{thm:algebraicsquare} shows that $K$ admits a filling pair of the form $( log(n)\cdot \overline{f}(n),g(n))$, with $\overline{f}$ the super-additive closure of $f(n)\asymp n^2+\sum_{i=1}^4 f_i(n)$ and $g(n)\asymp n+ \sum_{i=1}^4 g_i(n)$.
\end{corollary}

The general strategy will be the same as in the triangle method, except that we will replace triangles by squares and produce a different kind of filling for these squares. 
Given a loop in the Cayley graph labelled by a null-homotopic word with respect to a particular generating set for $K$, we start by subdividing its set of vertices into suitable subsets of 4-tuples. We will then construct spanning squares for these 4-tuples and glue them together to obtain a filling for our loop. As before each square will come with a decomposition into regions, such that their boundary words will be of length controlled by the perimeter of the square and lie in a group of the form $\ZZ^m, G_i\times\ZZ^l, G_i\times\ZZ^k$ or $G_i\times G_j$. Piecing these fillings together will thus allow us to obtain the desired bound on the Dehn function. 

The key innovation of this section is to take any four points $\a = (\a_1, \a_2, \a_3, \a_4)$, $\b = (\b_1, \b_2, \b_3, \b_4)$, $\c = (\c_1, \c_2, \c_3, \c_4)$ and $\d = (\d_1, \d_2, \d_3, \d_4)$ in $K$ and construct from them a square as in Figure \ref{fig:algebraicsquare} whose perimeter is controlled by $d(\a, \b), d(\b, \c), d(\c, \d)$ and $d(\a, \d)$, where $d$ is the product metric on $G_1\times G_2\times G_3\times G_4$. To do so we will make fundamental use of the fact that each of the sequences in Theorem \ref{thm:algebraicsquare} $P$-splits.
Once we have produced such a filling, the remainder of the argument will be similar to the argument given in Section \ref{sec:trianglemethod} and \cite{CarFor-17}.

As before the groups labelling the bounded regions in Figure \ref{fig:algebraicsquare} are given via explicit generating sets and the edges on their boundaries will be labelled by words in these generating sets. As in Section \ref{sec:trianglemethod}, we thus begin by defining these generating sets.

Let $W_i$ be a generating set for $G_i$. 
Let $P = \{A, B\}$ and $s_i^1, s_i^2$ be maps defining the $P$-splitting of $G_i\to \ZZ^m$, where $A\cong \ZZ^k$ and $B\cong \ZZ^l$.  
Let $\mathcal{A} = \{a_1, \dots, a_k\}$ be a basis for $A$ and $\mathcal{B} = \{b_1, \dots, b_l\}$ be a basis for $B$. Note that $\mathcal{A}\cup \mathcal{B}$ is a basis for $\ZZ^m$.
We then fix a finite generating set $Y_i\cup Z_i^1\cup Z_i^2$ of $G_i$, by choosing $Z_i^1 = \{s_i^1(a_1), \dots, s_i^1(a_k)\}$, $Z_i^2=\{s_i^2(b_1), \dots, s_i^2(b_l)\}$ and $Y_i = \{xs_i^1(p_A(\phi_i(x)))^{-1}s_i^2(p_B(\phi_i(x)))^{-1}\mid x\in W_i\}\subset \ker(\phi_i)$, where $p_A:\ZZ^m\to A$ and $p_B:\ZZ^m\to B$ are the canonical projections. 

\begin{notation*}
	In Lemma \ref{lem:squaresubgroups} we will define 14 sets which collectively generate $K$. 
	The reader should have the following in mind while looking at this notation: 
	
	$T_{ij}$ is a set which generates the kernel of $\phi$ restricted to $s_i^1(\ZZ^k)\times s_j^1(\ZZ^k)$. 
	Similarly, $U_{ij}$ is a set which generates the kernel of $\phi$ restricted to $s_i^2(\ZZ^l)\times s_j^2(\ZZ^l)$.
	Finally, $V_i$ is the image of $Y_i$ in $K$. 
\end{notation*}

We recommend that during a first reading the reader only skims the following technical lemma, where we define all of the generating sets and various subgroups generated by their unions, and only refers back to it as needed during the construction of the spanning square.

\begin{lemma}\label{lem:squaresubgroups}
	Let $Y_i, Z_i^j$ be as above. 
	Define the following subsets of $G_1\times G_2\times G_3\times G_4$.
	\begin{itemize}
		\item $T_{12} = \{(s_1^1(a), s_2^1(a)^{-1}, e, e)\mid a\in \mathcal{A}\}$,
		\item $T_{13} = \{(s_1^1(a), e, s_3^1(a)^{-1}, e)\mid a \in \mathcal{A}\}$,
		\item $T_{14} = \{(s_1^1(a), e, e, s_4^1(a)^{-1})\mid a\in \mathcal{A}\}$,
		\item $T_{23} = \{(e, s_2^1(a), s_3^1(a)^{-1}, e)\mid a\in \mathcal{A}\}$,
		\item $T_{24} = \{(e, s_2^1(a), e, s_4^1(a)^{-1})\mid a\in \mathcal{A}\}$,
		\item $T_{34} = \{(e, e, s_3^1(a), s_4^1(a)^{-1})\mid a\in \mathcal{A}\}$,
		\item $U_{12} = \{(s_1^2(b), s_2^2(b)^{-1}, e, e)\mid b\in \mathcal{B}\}$,
		\item $U_{13} = \{(s_1^2(b), e, s_3^2(b)^{-1}, e)\mid b\in \mathcal{B}\}$,
		\item $U_{14} = \{(s_1^2(b), e, e, s_4^2(b)^{-1})\mid b\in \mathcal{B}\}$,
		\item $U_{23} = \{(e, s_2^2(b), s_3^2(b)^{-1}, e)\mid b\in \mathcal{B}\}$,
		\item $U_{24} = \{(e, s_2^2(b), e, s_4^2(b)^{-1})\mid b\in \mathcal{B}\}$,
		\item $U_{34} = \{(e, e, s_3^2(b), s_4^2(b)^{-1})\mid b\in \mathcal{B}\}$,
		\item $V_1 = \{(y, e, e, e)\mid y\in Y_1\}$, 
		\item $V_2 = \{(e, y, e, e)\mid y\in Y_2\}$, 
		\item $V_3 = \{(e, e, y, e)\mid y\in Y_3\}$,
		\item $V_4 = \{(e, e, e, y)\mid y\in Y_4\}$.
	\end{itemize}
	Let $T_{ijk} = T_{ij}\cup T_{jk}\cup T_{ik}$ and $U_{ijk} = U_{ij}\cup U_{jk}\cup U_{ik}$.
	Let $T = \cup_{i,j} T_{ij}$, $U = \cup_{i, j} U_{ij}$, $V = \cup_i V_i$. 
	Then the following hold: 
	\begin{itemize}
		\item $\langle T\cup U\cup V\rangle  = \ker(\phi)$,
		\item $\langle T_{134}\cup U_{234}\cup V_3\cup V_4\rangle\cong G_3\times G_4$, 
		\item $\langle T_{123}\cup U_{234}\cup V_2\cup V_3\rangle\cong G_2\times G_3$, 
		\item $\langle T_{123}\cup U_{124}\cup V_1\cup V_2\rangle\cong G_1\times G_2$, 
		\item $\langle T_{124}\cup U_{234}\cup V_2\cup V_4\rangle\cong G_2\times G_4$, 
		\item $\langle T_{124}\cup U_{123}\cup V_1\cup V_2\rangle\cong G_1\times G_2$, 
		\item $\langle T_{124}\cup U_{134}\cup V_1\cup V_4\rangle\cong G_1\times G_4$, 
		\item $\langle T_{234}\cup U_{134}\cup V_3\cup V_4\rangle\cong G_3\times G_4$, 
		\item $\langle T_{123}\cup U_{134}\cup V_1\cup V_3\rangle\cong G_1\times G_3$, 
		\item $\langle T_{12}\cup U_{234}\cup V_2\rangle\cong G_2\times \ZZ^{l}$, 
		\item $\langle T_{124}\cup U_{34}\cup V_4\rangle\cong G_4\times \ZZ^{k}$, 
		\item $\langle T_{12}\cup U_{134}\cup V_1\rangle\cong G_1\times \ZZ^{l}$, 
		\item $\langle T_{123}\cup U_{34}\cup V_3\rangle\cong G_3\times \ZZ^{k}$, 
		\item $\langle T_{12}\cup U_{34}\rangle\cong \ZZ^{m}$. 
	\end{itemize}
\end{lemma}
\begin{proof}
	For the first isomorphism observe that $T\cup U \cup V \subset K$ by definition. Thus, let $(g_1,g_2,g_3,g_4)\in K$. Applying elements from $T\cup U \cup V$, we obtain an element of the form $(e,e,e,h)$. Using elements from $T_{24}\cup U_{34}\cup V_4$ we then get an element of the form $(e,\left(s_2^1(p_A(\phi_4(h))\right)^{-1},\left(s_3^2(p_B(\phi_4(h))\right)^{-1},e)=(e,e,e,e)$, since $\phi_4(h)=e$.
	
	For the others a similar strategy can be employed. 
	For instance, looking at the second one, let $H = \langle T_{134}\cup U_{234}\cup V_3\cup V_4\rangle$. 
	The natural projection $H\to G_3\times G_4$ is surjective. 
	The kernel of this surjection is $\{(g_1, g_2, g_3, g_4)\in H\mid g_3 = g_4 = e\}$. 
	However, we see that if $(g_1, g_2, g_3, g_4)\in H$, then $g_1 = s_1^1(p_A(\phi_3(g_3)) + p_A(\phi_4(g_4)))^{-1}$ and $g_2=s_2^2(p_B(\phi_3(g_3))+p_B(\phi_4(g_4)))^{-1}$. 
	Thus if $g_3 = g_4 = e$, then also $g_1=g_2=e$, implying that the projection $H\to G_3\times G_4$ is an isomorphism.
	
	By projecting to appropriate factors the same reasoning can be applied to obtain the other 12 isomorphisms. To give two examples, for $\langle T_{12}\cup U_{234}\cup V_2\rangle$ we can consider the projection to $G_2\times G_3$, while for $\langle T_{12}\cup U_{34}\rangle$ we can consider the projection to $G_1\times G_3$.
\end{proof}

As in Section \ref{sec:trianglemethod}, we may assume that $f$ is chosen in its equivalence class such that it satisfies the conditions in Theorem \ref{thm:algebraicsquare} and bounds from above the Dehn functions of the groups $G_i\times G_j$, $\ZZ^m$, $G_i\times \ZZ^l$, $G_i\times \ZZ^k$ and $G_1\times G_2\times G_3\times G_4$ with respect to fixed choices of presentations for the given generating sets.

\subsection{Constructing spanning Squares}
\label{sec:Constr-algspanning-squares}
We shall start by giving the construction of the square shaped loops in $K$ which we will use in the proof of Theorem \ref{thm:algebraicsquare}, postponing the remainder of the argument until later. 

Let $\a = (\a_1, \a_2, \a_3, \a_4), \b = (\b_1, \b_2, \b_3, \b_4), \c = (\c_1, \c_2, \c_3, \c_4)$ and $\d = (\d_1, \d_2, \d_3, \d_4)$ be four elements of $K$.

Figure \ref{fig:algebraicsquare} provides a depiction of the square that we will construct to fill our loops. It consists of 17 different regions. All of these regions will be contained in subgroups of $K$ obtained by using the generating sets from Lemma \ref{lem:squaresubgroups}, which were induced by the $P$-splittings of the factors.
In particular, each region will be isomorphic to a group of the form $G_i\times G_j, G_i\times \ZZ^k, G_i\times \ZZ^l$ or $\ZZ^m$. 
This will allow us to bound the area of each region from above by $f(n)$, where $n$ is the perimeter of the region. 
Thus it will be crucial to construct all edges in Figure \ref{fig:algebraicsquare} in a way that allows us to control their length.
Moreover, our construction will take care of the fact that edges will be chosen so that they lie in both adjacent regions. 
For all of them their length will be linearly bounded in terms of the ``perimeter'' of the big boundary square.

\begin{remark}
\label{rem:existence-uniqueness-square}
	For many vertices we will label two entries by $\cdot$ together with a subscript. 
	These vertices represent points which differ from the original vertices by elements of $s_i^1(\ZZ^k)$ or $s_i^2(\ZZ^l)$. 
	The subscripts on the vertex are to keep track of which part of the splitting they belong to. 
	For instance, $(\cdot_1, \cdot_2, \b_3, \a_4)$ denotes the unique point in $K$ where the first coordinate is in $\a_1s_1^1(\ZZ^k)$, the second in $\a_2s_2^2(\ZZ^l)$, the third is $\b_3$ and the fourth is $\a_4$.
	Wherever we use this notation, the elements will be unique and their existence will be guaranteed; both will always follow from the fact that $P$ is a factoring. 
\end{remark}
\begin{figure}
	\begin{tikzpicture}[scale = 0.35]
	\draw (0, 0) -- (0, 40) -- (40, 40) -- (40, 0) --(0, 0);
	\draw (10, 0) -- (10, 40);
	\draw (30, 0) -- (30, 40);
	\draw (0, 10) -- (40, 10);
	\draw (0, 30) -- (40, 30);
	\draw (20, 0) -- (20, 10);
	\draw (0, 20) -- (10, 20);
	\draw (40, 20) -- (30, 20);
	\draw (20, 40) -- (20, 30);
	\draw (20, 10) -- (10, 20) -- (20, 30) -- (30, 20) -- (20, 10);

	\draw (5, 5.5) node   {\scriptsize $\langle T_{123}\cup U_{124}\cup V_1\cup V_2\rangle$};
	\draw (15, 5.5) node  {\scriptsize $\langle T_{123}\cup U_{134}\cup V_1\cup V_3\rangle$};
	\draw (25, 5.5) node  {\scriptsize $\langle T_{123}\cup U_{134}\cup V_1\cup V_3\rangle$};
	\draw (35, 5.5) node  {\scriptsize $\langle T_{234}\cup U_{134}\cup V_3\cup V_4\rangle$};
	
	\draw (5, 35.5) node  {\scriptsize $\langle T_{134}\cup U_{234}\cup V_3\cup V_4\rangle$};
	\draw (15, 35.5) node {\scriptsize $\langle T_{124}\cup U_{234}\cup V_2\cup V_4\rangle$};
	\draw (25, 35.5) node {\scriptsize $\langle T_{124}\cup U_{234}\cup V_2\cup V_4\rangle$};
	\draw (35, 35.5) node {\scriptsize $\langle T_{124}\cup U_{123}\cup V_1\cup V_2\rangle$};
	
	\draw (5, 15.5) node  {\scriptsize $\langle T_{123}\cup U_{234}\cup V_2\cup V_3\rangle$};
	\draw (14.5, 14.5) node[rotate = -45] {$\langle T_{123}\cup U_{34}\cup V_3\rangle$};
	\draw (25.5, 14.5) node[rotate = 45] {$\langle T_{12}\cup U_{134}\cup V_1\rangle$};
	\draw (35, 15.5) node {\scriptsize $\langle T_{124}\cup U_{134}\cup V_1\cup V_4\rangle$};
	
	\draw (5, 25.5) node  {\scriptsize $\langle T_{123}\cup U_{234}\cup V_2\cup V_3\rangle$};
	\draw (13.5, 26.5) node[rotate = 45] {$\langle T_{12}\cup U_{234}\cup V_2\rangle$};
	\draw (26.5, 26.5) node[rotate = -45] {$\langle T_{124}\cup U_{34}\cup V_4\rangle$};
	\draw (35, 25.5) node {\scriptsize $\langle T_{124}\cup U_{134}\cup V_1\cup V_4\rangle$};
	
	\draw (20, 20.5) node {$\langle T_{12}\cup U_{34}\rangle$};
	
	\draw (5, 4.5) node   {\scriptsize $\cong G_1\times G_2$};
	\draw (15, 4.5) node  {\scriptsize $\cong G_1\times G_3$};
	\draw (25, 4.5) node  {\scriptsize $\cong G_1\times G_3$};
	\draw (35, 4.5) node  {\scriptsize $\cong G_3\times G_4$};
	
	\draw (5, 34.5) node  {\scriptsize $\cong G_3\times G_4$};
	\draw (15, 34.5) node {\scriptsize $\cong G_2\times G_4$};
	\draw (25, 34.5) node {\scriptsize $\cong G_2\times G_4$};
	\draw (35, 34.5) node {\scriptsize $\cong G_1\times G_2$};
	
	\draw (5, 14.5) node  {\scriptsize $\cong G_2\times G_3$};
	\draw (13.5, 13.5) node[rotate = -45] {$\cong G_3\times \ZZ^k$};
	\draw (26.5, 13.5) node[rotate = 45] {$\cong G_1\times \ZZ^l$};
	\draw (35, 14.5) node {\scriptsize $\cong G_1\times G_4$};
	
	\draw (5, 24.5) node  {\scriptsize $\cong G_2\times G_3$};
	\draw (14.5, 25.5) node[rotate = 45] {$\cong G_2\times \ZZ^l$};
	\draw (25.5, 25.5) node[rotate = -45] {$\cong G_4\times \ZZ^k$};
	\draw (35, 24.5) node {\scriptsize $\cong G_1\times G_4$};
	
	\draw (20, 19.5) node {$\cong \ZZ^m$};
	
	\fill (0,   0) circle(1em) node[below right] {$(\d_1    , \d_2    , \d_3    , \d_4)$};
	\fill (0,  10) circle(1em) node[below right] {$(\a_1    , \d_2    , \cdot_1, \cdot_2)$};
	\fill (0,  20) circle(1em) node[below right] {$(\cdot_1, \d_2    , \a_3    , \cdot_2)$};
	\fill (0,  30) circle(1em) node[below right] {$(\cdot_1, \cdot_2, \a_3    , \d_4)$};
	\fill (0,  40) circle(1em) node[below right] {$(\a_1    , \a_2    , \a_3    , \a_4)$};
	\fill (10,  0) circle(1em) node[below right] {$(\d_1    , \c_2    , \cdot_1, \cdot_2)$};
	\fill (10, 10) circle(1em) node[below right] {$(\a_1    , \c_2    , \cdot_1, \cdot_2)$};
	\fill (10, 20) circle(1em) node[above left ] {$(\cdot_1, \c_2    , \b_3    , \cdot_2)$};
	\fill (10, 30) circle(1em) node[above right] {$(\cdot_1, \cdot_2, \b_3    , \d_4)$};
	\fill (10, 40) circle(1em) node[below right] {$(\cdot_1, \cdot_2, \b_3    , \a_4)$};
	\fill (30,  0) circle(1em) node[below right] {$(\cdot_2, \cdot_1, \c_3    , \d_4)$};
	\fill (30, 10) circle(1em) node[below right] {$(\cdot_2, \cdot_1, \b_3    , \d_4)$};
	\fill (30, 20) circle(1em) node[below right] {$(\a_1    , \cdot_1, \cdot_2, \d_4)$};
	\fill (30, 30) circle(1em) node[below right] {$(\a_1    , \c_2    , \cdot_2, \cdot_1)$};
	\fill (30, 40) circle(1em) node[below right] {$(\a_1    , \b_2    , \cdot_2, \cdot_1)$};
	\fill (40,  0) circle(1em) node[above left ] {$(\c_1    , \c_2    , \c_3    , \c_4)$};
	\fill (40, 10) circle(1em) node[above left ] {$(\cdot_2, \cdot_1, \b_3    , \c_4)$};
	\fill (40, 20) circle(1em) node[above left ] {$(\b_1    , \cdot_1, \cdot_2, \c_4)$};
	\fill (40, 30) circle(1em) node[above left ] {$(\b_1    , \c_2    , \cdot_2, \cdot_1)$};
	\fill (40, 40) circle(1em) node[above left ] {$(\b_1    , \b_2    , \b_3    , \b_4)$};
	\fill (20,  0) circle(1em) node[below right] {$(\d_1    , \cdot_1, \c_3    , \cdot_2)$};
	\fill (20, 10) circle(1em) node[below right] {$(\a_1    , \cdot_1, \b_3    , \cdot_2)$};
	\fill (20, 30) circle(1em) node[above right] {$(\cdot_1, \c_2    , \cdot_2, \d_4)$};
	\fill (20, 40) circle(1em) node[below right] {$(\cdot_1, \b_2    , \cdot_2, \a_4)$};
	\end{tikzpicture}
	\caption{A diagram of the spanning square. The regions of the square are labeled by the groups generated by elements used to traverse the perimeter.}
	\label{fig:algebraicsquare}
\end{figure}
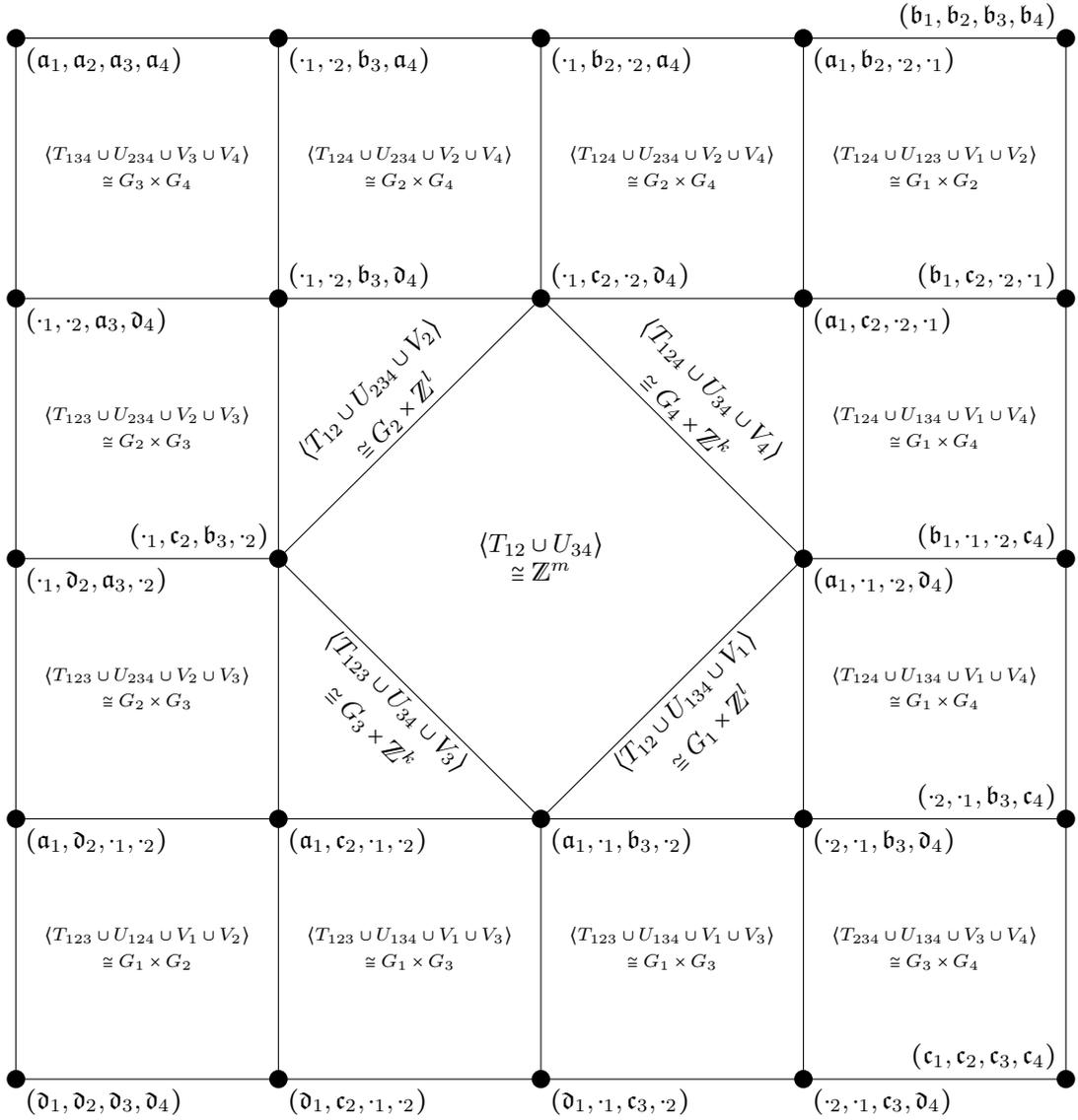

We will now provide a detailed construction of the square in Figure \ref{fig:algebraicsquare}. We start by observing that there is a natural action of the dihedral group on the unlabelled square. If two edges are in the same orbit we can bound their lengths using the same reasoning. We thus reduce to a fundamental domain for this action and focus on the vertices and edges in Figure \ref{fig:algedges}. We will denote by $d_i$ the word metric on $G_i$ with respect to the generating sets $Y_i\cup Z_i^1\cup Z_i^2$ for $1\leq i \leq 4$.

\begin{figure}
	\center
	\begin{tikzpicture}[scale = 0.6]
	\draw (0, 0) -- (20, 0) -- (20, 10) -- (10, 20) -- (0, 20) -- (0, 0);
	\draw (0, 10) -- (10, 10) -- (10, 0);
	\draw (20, 10) -- (10, 10) -- (10, 20);
	\fill (0,   0) circle(0.5em) node[below right] {$(\d_1  , \d_2  , \d_3  , \d_4)$};
	\fill (0,  10) circle(0.5em) node[below right] {$(\a_1  , \d_2  , \cdot_1, \cdot_2)$};
	\fill (0,  20) circle(0.5em) node[below right] {$(\cdot_1, \d_2  , \a_3  , \cdot_2)$};
	\fill (10,  0) circle(0.5em) node[below right] {$(\d_1  , \c_2  , \cdot_1, \cdot_2)$};
	\fill (10, 10) circle(0.5em) node[below right] {$(\a_1  , \c_2  , \cdot_1, \cdot_2)$};
	\fill (10, 20) circle(0.5em) node[      right] {$(\cdot_1, \c_2  , \b_3  , \cdot_2)$};
	\fill (20,  0) circle(0.5em) node[below left ] {$(\d_1  , \cdot_1, \c_3  , \cdot_2)$};
	\fill (20, 10) circle(0.5em) node[below left ] {$(\a_1  , \cdot_1, \b_3  , \cdot_2)$};
	
	\fill (5, 0) node[above] {$d_2(\d_2, \c_2)$};
	\fill (-0.5, 5) node[rotate = 90] {$d_1(\d_1, \a_1)$};
	\fill (5, 10) node[above] {$d_2(\d_2, \c_2)$};
	\fill (5, 20) node[above] {$\leq d_2(\d_2, \c_2) + d_3(\a_3, \b_3)$};
	\fill (19.5, 5) node[rotate = 90] {$\leq d_1(\a_1, \d_1) + d_3(\b_3, \c_3)$};
	\fill (9.5, 5) node[rotate = 90] {$d_1(\a_1, \d_1)$};
	\fill (15, 0) node[above] {$\leq d_2(\c_2, \d_2) + d_3(\c_3, \d_3)$};
	\fill (-0.5, 15) node[rotate = 90] {$\leq d_1(\a_1, \d_1) + d_3(\a_3, \d_3)$};
	\fill (15, 10) node[above] {$\leq d_1(\a_1, \d_1) + d_2(\c_2, \d_2) + d_3(\b_3, \d_3)$};
	\fill (9.5, 15) node[rotate = 90] {$\leq d_2(\c_2, \d_2) + d_1(\a_1, \d_1) + d_3(\b_3, \d_3)$};
	\fill (15.35, 15.35) node[rotate = -45] {$\leq 2\left(d_2(\c_2, \d_2) + d_1(\a_1, \d_1) + d_3(\b_3, \d_3)\right)$};
	
	\draw (5, 5) node   {$\langle T_{123}\cup U_{124}\cup V_1\cup V_2\rangle$};
	\draw (5, 4) node   {$\cong G_{1}\times G_2$};
	\draw (15, 5) node  {$\langle T_{123}\cup U_{134}\cup V_1\cup V_3\rangle$};
	\draw (15, 4) node   {$\cong G_{1}\times G_3$};
	\draw (5, 15) node  {$\langle T_{123}\cup U_{234}\cup V_2\cup V_3\rangle$};
	\draw (5, 14) node   {$\cong G_{2}\times G_3$};
	\draw (14, 14) node[rotate = -45] {$\langle T_{123}\cup U_{34}\cup V_3\rangle$};
	\draw (13.5, 13.5) node[rotate = -45]   {$\cong G_3\times \ZZ^k$};

	\end{tikzpicture}
	\caption{}
	\label{fig:algedges}
\end{figure}

We begin with the edge between the vertices labelled $(\d_1, \d_2, \d_3, \d_4)$ and $(\d_1, \c_2, \cdot_1, \cdot_2) = (\d_1, \c_2, \d_3', \d_4')$. 
To obtain it, we change the second coordinate as required by applying $d_2(\c_2,\d_2)$ generators from $V_2\cup T_{23}\cup U_{24}$. 
Since the only generators that change the third coordinate are from $T_{23}$ we see that $\d_3\cdot (\d_3')^{-1}\in s_3^1(\ZZ^k)$ and this element has distance $\leq d_2(\c_2, \d_2)$ from the identity for the generators $Z_3^1$. 
Similarly, $\d_4\cdot (\d_4')^{-1}\in s_4^2(\ZZ^l)$ has distance $\leq d_2(\c_2, \d_2)$ from the identity for the generators $Z_4^2$. 

Similarly, we can find suitable paths of length $d_1(\a_1, \d_1)$ between $(\d_1, \d_2, \d_3, \d_4)$ and $(\a_1, \d_2, \cdot_1, \cdot_2)$, of length $ d_1(\a_1, \d_1)$ between $(\d_1, \c_2, \cdot_1, \cdot_2)$ and $(\a_1, \c_2, \cdot_1, \cdot_2)$, and of length $ d_2(\c_2, \d_2)$ between $(\a_1, \d_2, \cdot_1, \cdot_2)$ and $(\a_1, \c_2, \cdot_1, \cdot_2)$. These are in the respective generating sets $V_2\cup T_{23}\cup U_{24}$, $V_1\cup T_{13}\cup U_{14}$ and $V_1\cup T_{13}\cup U_{14}$. Due to the uniqueness of the third and fourth coordinate in terms of the first two coordinates (see Remark \ref{rem:existence-uniqueness-square}) this completes the square labelled $\langle T_{123}\cup U_{124}\cup V_1 \cup V_2\rangle$. \vspace{.3cm}

Now consider the vertices labelled $(\d_1, \c_2, \cdot_1, \cdot_2) = (\d_1, \c_2, \d_3', \d_4')$ and $(\d_1, \cdot_1, \c_3, \cdot_2)$. 
Using generators from $V_3\cup U_{34}\cup T_{23}$ we obtain a path of length $d_3(\c_3, \d_3')$ between them. 
Using the triangle inequality we see that this requires at most $d_3(\c_3, \d_3) + d_3(\d_3, \d_3')$ generators. 
By the above we have that $d_3(\d_3, \d_3')\leq d_2(\c_2, \d_2)$. 
Thus the length of the path labelling this edge is $\leq d_3(\c_3, \d_3) + d_2(\c_2, \d_2)$. 

Similar reasoning shows that $(\a_1, \d_2, \cdot_1, \cdot_2)$ and $(\cdot_1, \d_2, \a_3, \cdot_2)$ are connected by a path of length $\leq d_1(\a_1, \d_1) + d_3(\a_3, \d_3)$ in $V_3\cup U_{34}\cup T_{13}$.\vspace{.3cm}

Next we construct the edge between $(\d_1, \cdot_1, \c_3, \cdot_2)$ and $(\a_1, \cdot_1, \b_3, \cdot_2)$.
For this we use generators from $V_1\cup V_3\cup T_{123}\cup U_{134}$.
To change the first coordinate we require $d_1(\a_1, \d_1)$ generators from $V_1\cup T_{12}\cup U_{14}$ and to change the third coordinate we require $d_3(\b_3, \c_3)$ generators from $V_3\cup T_{13}\cup U_{34}$. 
Thus, our path between the two vertices uses $d_1(\a_1, \d_1) + d_3(\b_3, \c_3)$ generators.
Similar reasoning shows that the vertices $(\cdot_1, \d_2, \a_3, \cdot_2)$ and $(\cdot_1, \c_2, \b_3, \cdot_2)$ can be connected using a path of length $ d_3(\a_3, \b_3) + d_2(\c_2, \d_2)$ in the generators $V_2\cup V_3 \cup T_{123}\cup U_{234}$.  \vspace{.3cm}

We are now left with constructing the paths labelling the edges of the triangle in Figure \ref{fig:algedges}. 
We will start by considering the edge with vertices $(\a_1, \c_2, \cdot_1, \cdot_2) = (\a_1, \c_2, \d_3'', \cdot_2)$ and $(\a_1, \cdot_1, \b_3, \cdot_2)$. 
To move between these vertices we use generators from $V_3\cup T_{23}\cup U_{34}$.
Once again, this requires $d_3(\d_3'', \b_3)\leq d_3(\b_3, \d_3) + d_3(\d_3, \d_3'')$ generators. 
We obtained $\d_3''\in \d_3\cdot s_3^1(\ZZ^k)$ from $\d_3$ by first taking a path of length $d_2(\c_2, \d_2)$ from $(\d_1, \d_2, \d_3, \d_4)$ to $(\d_1, \c_2, \cdot_1, \cdot_2)$ and then a path of length $d_1(\a_1, \d_1)$ from $(\d_1, \c_2, \cdot_1, \cdot_2)$ to $(\a_1, \c_2, \cdot_1, \cdot_2)$. 
Thus $d_3(\d_3, \d_3'')\leq d_1(\a_1, \d_1) + d_2(\c_2, \d_2)$. 

Analogous arguments provide a path of length $\leq d_3(\b_3,\d_3)+d_2(\c_2,\d_2)+d_1(\a_1,\d_1)$ between $(\a_1, \c_2, \cdot_1, \cdot_2)$ and $(\cdot_1, \c_2, \b_3, \cdot_2)$ using elements from $V_3\cup T_{13}\cup U_{34}$.\vspace{.3cm}

Finally we construct the diagonal edge of the triangle, connecting the vertices $(\cdot_1, \c_2, \b_3, \cdot_2) = (\a_1', \c_2, \b_3, \d_4'')$ and $(\a_1,\c_2',\b_3,\d_4''')=(\a_1,\cdot_1, \b_3,\cdot_2)$ using generators from $T_{12}$. For this we first observe that $\d_4'' =\d_4'''$, because this coordinate is uniquely determined by the projection of $\phi_3(\b_3)$ to $B\leq \ZZ^m$. 
The only generators in $V_3\cup T_{123}\cup U_{23}$ that change the first and second coordinate are from $T_{123}$. Moreover, the first coordinate is uniquely determined by the second and third coordinates and the second coordinate is uniquely determined by the first and third coordinates. We deduce that there is a geodesic path in generators from $Z_1^1$ connecting $\a_1$ to $\a_1'$ and that any such path lifts to a path in generators from $T_{12}$ connecting the two vertices at the ends of the diagonal edge of the triangle. Its length is $2\left( d_3(\b_3, \d_3) + d_1(\a_1, \d_1) + d_2(\c_2, \d_2)\right)$, since it is bounded above by the number of generators from $T_{123}$ used in the paths labelling the other two edges of the triangle.

The construction of all other paths between vertices in Figure \ref{fig:algebraicsquare} and the upper bounds on their lengths can be obtained by very similar arguments using the symmetries of the square. This completes our construction of the filling square.\vspace{.3cm}

Let $r_\a = \max_i\{d_i(\a_i, \b_i)\}$, $r_\b = \max_i\{d_i(\b_i, \c_i)\}, r_\c = \max_i\{d_i(\c_i, \d_i)\}$ and $r_\d = \max_i\{d_i(\a_i, \d_i)\}$. 
We denote $U:=r_\a + r_\b + r_\c + r_\d$ and observe that $U$ is bounded above by $D = d(\a, \b) + d(\b, \c) + d(\c, \d) + d(\a, \d)$, which we refer to as the {\em perimeter} of the square. 
This will permit us to bound the area of our square between four vertices purely in terms of their distances. 
Indeed, using triangle inequalities in the $X_i$, we deduce readily that all words labelling the edges of the square have length $\leq 4 U$ and that each boundary loop labelling a bounded region has perimeter at most $12U$. 
By definition of $f$, every region thus admits a filling of area $\leq f(12U)$.
In particular, the square admits a filling of area $\leq 17\cdot f(12U)$. 

\subsection{Tessellating a loop by squares}

We now turn to the remainder of the proof of Theorem \ref{thm:algebraicsquare}.
The paths on the boundary of the square are not chosen in a canonical way. 
When constructing our filling of a loop we thus need to take care of two things: 
{\em (i)} that squares can be glued together and 
{\em (ii)} that we can complete the filling by squares to a filling of our initial loop. 
(i) does not pose any issues in view of the fact that it will be evident that we will be able to choose the same edge path on adjacent squares. 
To resolve (ii) we require the following auxiliary result, where we call the path along a side of the square a {\em spanning path} (adapting the terminology from \cite{CarFor-17}).

\begin{lemma}\label{lem:bigons}
	If two vertices have distance 1 in $K$, then the spanning path between them has length $\leq 4$. 
\end{lemma}
\begin{proof}
	Suppose that the two vertices are $\a$ and $\b$. 
	The spanning path between them has length bounded by $2d_1(\a_1, \b_1) + d_2(\a_2, \b_2) + d_3(\a_3, \b_3) + 2d_4(\a_4, \b_4)$. 
	Since $\a$ and $\b$ differ by a single generator of $K$, at most two summands are non-zero, in which case they are 1 or 2. This provides the desired upper bound. 
\end{proof}

The remainder of the proof of \cref{thm:algebraicsquare} is parallel to the one in Section \ref{sec:trianglemethod}, up to replacing triangles by squares and adjusting everything else accordingly. 

\begin{proof}[Proof of \cref{thm:algebraicsquare}]
	Let $\hat{\gamma}$ be the a loop in the Cayley graph of $K$ of length $n\geq 4$, parametrized by its length; it corresponds to a null-homotopic word in $T\cup U\cup V$ of the same length.
	We can find a $k$ such that $4\cdot 3^{k-1}\leq n\leq 4\cdot 3^{k}$. 
	As in Section \ref{sec:trianglemethod}, let $\gamma$ be the loop parametrized on the interval $\left[0,l\right]$, with $l = 4\cdot 3^k$, obtained by adding a trivial path to the end of $\hat{\gamma}$. 
	Since we have appended a trivial path $\gamma$ and $\hat{\gamma}$ have the same area. 
	We will show that $\gamma$ has area $\preccurlyeq f(n)$.
		\begin{figure}
			\center
			\begin{tikzpicture}[scale=6]

				\draw[black, thick, domain = 0:90] plot ({cos(\x)}, {sin(\x)});
				\draw[black, thick, domain = 180:270] plot ({cos(\x)}, {sin(\x)});
				\draw[black, thick, domain = 270:360] plot ({cos(\x)}, {sin(\x)});
				\draw[black, thick, domain = 90:180] plot ({cos(\x)}, {sin(\x)});
				
				\draw [black,thick,domain=269.613242082:180.0] plot ({1.0+1.0*cos(\x)}, {0.9999999999999999+1.0*sin(\x)});
				\draw [black,thick,domain=0.0:-90.3867579182] plot ({-1.0+1.0*cos(\x)}, {1.0+1.0*sin(\x)});
				\draw [black,thick,domain=90.3867579182:0.0] plot ({-1.0+1.0*cos(\x)}, {-0.9999999999999998+1.0*sin(\x)});
				\draw [black,thick,domain=180.0:90.3867579182] plot ({1.0+1.0*cos(\x)}, {-1.0000000000000002+1.0*sin(\x)});
				\draw [black,thick,domain=269.613242082:119.9982303945925] plot ({1.0+0.267949192431*cos(\x)}, {0.26794919243112264+0.267949192431*sin(\x)});
				\draw [black,thick,domain=299.9982303945925:149.99911519729628] plot ({0.7320508075688774+0.267949192431*cos(\x)}, {0.7320508075688773+0.267949192431*sin(\x)});
				\draw [black,thick,domain=329.99911519729625:180.0] plot ({0.26794919243112286+0.267949192431*cos(\x)}, {1.0+0.267949192431*sin(\x)});
				\draw [black,thick,domain=359.99999999999994:210.00088480270372] plot ({-0.26794919243112236+0.267949192431*cos(\x)}, {1.0000000000000002+0.267949192431*sin(\x)});
				\draw [black,thick,domain=30.000884802703702:-119.99823039459253] plot ({-0.7320508075688771+0.267949192431*cos(\x)}, {0.7320508075688776+0.267949192431*sin(\x)});
				\draw [black,thick,domain=60.001769605407496:-90.3867579182] plot ({-1.0+0.267949192431*cos(\x)}, {0.26794919243112264+0.267949192431*sin(\x)});
				\draw [black,thick,domain=89.99734559188883:-60.001769605407446] plot ({-0.9999999999999999+0.267949192431*cos(\x)}, {-0.2679491924311232+0.267949192431*sin(\x)});
				\draw [black,thick,domain=119.99823039459247:-30.00088480270375] plot ({-0.732050807568878+0.267949192431*cos(\x)}, {-0.7320508075688767+0.267949192431*sin(\x)});
				\draw [black,thick,domain=149.99911519729622:0.0] plot ({-0.26794919243112303+0.267949192431*cos(\x)}, {-1.0+0.267949192431*sin(\x)});
				\draw [black,thick,domain=179.99999999999994:30.00088480270373] plot ({0.2679491924311222+0.267949192431*cos(\x)}, {-1.0000000000000002+0.267949192431*sin(\x)});
				\draw [black,thick,domain=210.00088480270378:60.00176960540744] plot ({0.7320508075688772+0.267949192431*cos(\x)}, {-0.7320508075688771+0.267949192431*sin(\x)});
				\draw [black,thick,domain=240.00176960540745:90.3867579182] plot ({1.0+0.267949192431*cos(\x)}, {-0.26794919243112303+0.267949192431*sin(\x)});
				\draw [black,thick,domain=269.613242082:99.99764052612346] plot ({1.0+0.0874886635259*cos(\x)}, {0.08748866352592465+0.0874886635259*sin(\x)});
				\draw [black,thick,domain=279.9976405261233:109.99793546035794] plot ({0.9696155060244162+0.0874886635259*cos(\x)}, {0.25980769180793667+0.0874886635259*sin(\x)});
				\draw [black,thick,domain=289.99793546035795:119.99823039459241] plot ({0.909769735547401+0.0874886635259*cos(\x)}, {0.4242325948433998+0.0874886635259*sin(\x)});
				\draw [black,thick,domain=299.99823039459255:129.99852532882704] plot ({0.8222810720214766+0.0874886635259*cos(\x)}, {0.5757674051565992+0.0874886635259*sin(\x)});
				\draw [black,thick,domain=309.9985253288271:139.9988202630618] plot ({0.7098078142164789+0.0874886635259*cos(\x)}, {0.7098078142164801+0.0874886635259*sin(\x)});
				\draw [black,thick,domain=319.9988202630616:149.9991151972963] plot ({0.5757674051565996+0.0874886635259*cos(\x)}, {0.8222810720214765+0.0874886635259*sin(\x)});
				\draw [black,thick,domain=329.99911519729625:159.9994101315308] plot ({0.4242325948434016+0.0874886635259*cos(\x)}, {0.9097697355474003+0.0874886635259*sin(\x)});
				\draw [black,thick,domain=339.9994101315309:169.99970506576534] plot ({0.2598076918079369+0.0874886635259*cos(\x)}, {0.9696155060244159+0.0874886635259*sin(\x)});
				\draw [black,thick,domain=349.99970506576545:180.0] plot ({0.08748866352592348+0.0874886635259*cos(\x)}, {1.0+0.0874886635259*sin(\x)});
				\draw [black,thick,domain=0.0:-169.99970506576545] plot ({-0.08748866352592354+0.0874886635259*cos(\x)}, {1.0+0.0874886635259*sin(\x)});
				\draw [black,thick,domain=10.000294934234619:-159.99941013153082] plot ({-0.25980769180793717+0.0874886635259*cos(\x)}, {0.9696155060244159+0.0874886635259*sin(\x)});
				\draw [black,thick,domain=20.000589868469106:-149.99911519729625] plot ({-0.42423259484340103+0.0874886635259*cos(\x)}, {0.9097697355474006+0.0874886635259*sin(\x)});
				\draw [black,thick,domain=30.000884802703634:-139.99882026306165] plot ({-0.5757674051565987+0.0874886635259*cos(\x)}, {0.8222810720214772+0.0874886635259*sin(\x)});
				\draw [black,thick,domain=40.0011797369383:-129.99852532882716] plot ({-0.709807814216479+0.0874886635259*cos(\x)}, {0.7098078142164799+0.0874886635259*sin(\x)});
				\draw [black,thick,domain=50.001474671172986:-119.99823039459261] plot ({-0.8222810720214766+0.0874886635259*cos(\x)}, {0.5757674051565991+0.0874886635259*sin(\x)});
				\draw [black,thick,domain=60.001769605407496:-109.99793546035795] plot ({-0.9097697355474006+0.0874886635259*cos(\x)}, {0.42423259484340087+0.0874886635259*sin(\x)});
				\draw [black,thick,domain=70.00206453964205:-99.99764052612346] plot ({-0.9696155060244159+0.0874886635259*cos(\x)}, {0.25980769180793717+0.0874886635259*sin(\x)});
				\draw [black,thick,domain=80.00235947387671:-90.3867579182] plot ({-1.0+0.0874886635259*cos(\x)}, {0.08748866352592373+0.0874886635259*sin(\x)});
				\draw [black,thick,domain=90.3867579182:-80.00235947387671] plot ({-1.0+0.0874886635259*cos(\x)}, {-0.08748866352592373+0.0874886635259*sin(\x)});
				\draw [black,thick,domain=99.9976405261234:-70.00206453964199] plot ({-0.9696155060244159+0.0874886635259*cos(\x)}, {-0.25980769180793745+0.0874886635259*sin(\x)});
				\draw [black,thick,domain=109.99793546035798:-60.0017696054075] plot ({-0.9097697355474007+0.0874886635259*cos(\x)}, {-0.42423259484340053+0.0874886635259*sin(\x)});
				\draw [black,thick,domain=119.99823039459253:-50.00147467117307] plot ({-0.8222810720214774+0.0874886635259*cos(\x)}, {-0.575767405156598+0.0874886635259*sin(\x)});
				\draw [black,thick,domain=129.99852532882716:-40.0011797369383] plot ({-0.7098078142164788+0.0874886635259*cos(\x)}, {-0.7098078142164801+0.0874886635259*sin(\x)});
				\draw [black,thick,domain=139.9988202630616:-30.000884802703695] plot ({-0.5757674051566003+0.0874886635259*cos(\x)}, {-0.822281072021476+0.0874886635259*sin(\x)});
				\draw [black,thick,domain=149.99911519729622:-20.000589868469092] plot ({-0.424232594843401+0.0874886635259*cos(\x)}, {-0.9097697355474006+0.0874886635259*sin(\x)});
				\draw [black,thick,domain=159.99941013153082:-10.000294934234546] plot ({-0.25980769180793656+0.0874886635259*cos(\x)}, {-0.9696155060244162+0.0874886635259*sin(\x)});
				\draw [black,thick,domain=169.99970506576545:-5.684341886080802e-14] plot ({-0.08748866352592416+0.0874886635259*cos(\x)}, {-0.9999999999999999+0.0874886635259*sin(\x)});
				\draw [black,thick,domain=180.0:10.000294934234544] plot ({0.08748866352592372+0.0874886635259*cos(\x)}, {-1.0+0.0874886635259*sin(\x)});
				\draw [black,thick,domain=190.0002949342346:20.000589868469152] plot ({0.259807691807937+0.0874886635259*cos(\x)}, {-0.969615506024416+0.0874886635259*sin(\x)});
				\draw [black,thick,domain=200.00058986846918:30.000884802703737] plot ({0.42423259484340065+0.0874886635259*cos(\x)}, {-0.9097697355474007+0.0874886635259*sin(\x)});
				\draw [black,thick,domain=210.00088480270364:40.00117973693835] plot ({0.575767405156599+0.0874886635259*cos(\x)}, {-0.8222810720214769+0.0874886635259*sin(\x)});
				\draw [black,thick,domain=220.00117973693824:50.00147467117287] plot ({0.7098078142164791+0.0874886635259*cos(\x)}, {-0.7098078142164799+0.0874886635259*sin(\x)});
				\draw [black,thick,domain=230.00147467117293:60.001769605407524] plot ({0.8222810720214767+0.0874886635259*cos(\x)}, {-0.5757674051565992+0.0874886635259*sin(\x)});
				\draw [black,thick,domain=240.00176960540747:70.00206453964199] plot ({0.9097697355474004+0.0874886635259*cos(\x)}, {-0.42423259484340126+0.0874886635259*sin(\x)});
				\draw [black,thick,domain=250.00206453964216:80.00235947387664] plot ({0.9696155060244163+0.0874886635259*cos(\x)}, {-0.2598076918079358+0.0874886635259*sin(\x)});
				\draw [black,thick,domain=260.00235947387654:90.3867579182] plot ({1.0+0.0874886635259*cos(\x)}, {-0.08748866352592465+0.0874886635259*sin(\x)});

				\node () at (1.03, 0.0) {$\a$};
				\node () at (1.0143519856, 0.178857622997) {$\b$};
				\node () at (0.967883399409, 0.352280747625) {$\c$};
				\node () at (0.892006165898, 0.515) {$\d$};
				\node () at (0.789025776413, 0.662071237977) {$\a$};
				\node () at (0.662071237977, 0.789025776413) {$\b$};
				\node () at (0.515, 0.892006165898) {$\c$};
				\node () at (0.352280747625, 0.967883399409) {$\d$};
				\node () at (0.178857622997, 1.0143519856) {$\a$};
				\node () at (6.30693101561e-17, 1.03) {$\b$};
				\node () at (-0.178857622997, 1.0143519856) {$\c$};
				\node () at (-0.352280747625, 0.967883399409) {$\d$};
				\node () at (-0.515, 0.892006165898) {$\a$};
				\node () at (-0.662071237977, 0.789025776413) {$\b$};
				\node () at (-0.789025776413, 0.662071237977) {$\c$};
				\node () at (-0.892006165898, 0.515) {$\d$};
				\node () at (-0.967883399409, 0.352280747625) {$\a$};
				\node () at (-1.0143519856, 0.178857622997) {$\b$};
				\node () at (-1.03, 1.26138620312e-16) {$\c$};
				\node () at (-1.0143519856, -0.178857622997) {$\d$};
				\node () at (-0.967883399409, -0.352280747625) {$\a$};
				\node () at (-0.892006165898, -0.515) {$\b$};
				\node () at (-0.789025776413, -0.662071237977) {$\c$};
				\node () at (-0.662071237977, -0.789025776413) {$\d$};
				\node () at (-0.515, -0.892006165898) {$\a$};
				\node () at (-0.352280747625, -0.967883399409) {$\b$};
				\node () at (-0.178857622997, -1.0143519856) {$\c$};
				\node () at (-1.89207930468e-16, -1.03) {$\d$};
				\node () at (0.178857622997, -1.0143519856) {$\a$};
				\node () at (0.352280747625, -0.967883399409) {$\b$};
				\node () at (0.515, -0.892006165898) {$\c$};
				\node () at (0.662071237977, -0.789025776413) {$\d$};
				\node () at (0.789025776413, -0.662071237977) {$\a$};
				\node () at (0.892006165898, -0.515) {$\b$};
				\node () at (0.967883399409, -0.352280747625) {$\c$};
				\node () at (1.0143519856, -0.178857622997) {$\d$};
			\end{tikzpicture}

			\caption{A Disc tessellated by squares and bigons.}
			\label{fig:hypsquare}
		\end{figure}
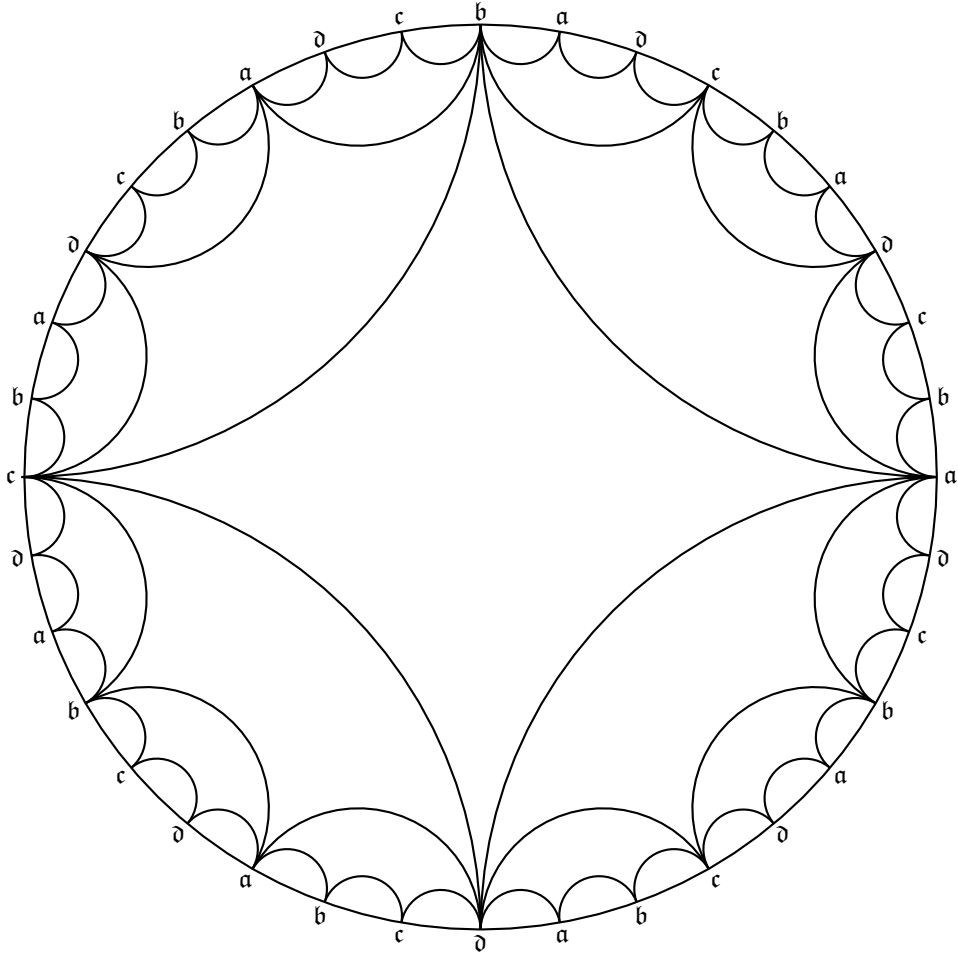
	
	Let $D$ be the disk shown in \cref{fig:hypsquare}. 
	It has $l$ vertices on the boundary, $l$ bigons adjacent to the boundary and $2\cdot 3^k - 1$ squares. 
	To each square we assign a {\em depth}: the central square is at depth 0 and its neighbours are at depth 1 and so on. 
	For $k\geq i\geq 1$ there are $4\cdot 3^{i-1}$ squares of depth $i$. 
	
	There is a labelling of the vertices of \cref{fig:hypsquare} by the set $\{\a, \b, \c, \d\}$ such that the boundary of each square reads $\a, \b, \c, \d$ or $\d, \c, \b, \a$. 
	To obtain such a labelling we proceed as follows. 
	Each edge is in exactly two squares, pick a labelling on the central square such that the boundary reads $\a,\b,\c,\d$.
	Label the vertices of squares of level $i$ by reflecting the labelling in the squares of level $i-1$ along the edge joining them. 

	Each square can now be filled with a spanning square for its four points on the boundary.
	The reflection technique used for the labelling allows us to choose the paths in the spanning squares so that they agree on edges in their intersection. More precisely, it enables us to choose the groups labelling the bounded regions in adjacent spanning squares so that they coincide in the regions along the shared edge. We then use that the choice of path in the Cayley graph between two vertices of the spanning square, say, $\a$ and $\d$, only depends on $\a$, $\d$ and the choice of groups labelling the adjacent bounded regions. In particular, it does not depend on $\b$ and $\c$. 
	
	Give each edge a {\em depth} by declaring it to be the minimum depth of squares adjacent to the edge.
	If $e$ is an edge of depth $i$, then its boundary vertices are at distance $\leq 3^{k-i}$.
	
	The central square has a perimeter bounded by $4\cdot 3^k$, and,  
	for each $k\geq i\geq 1$, the perimeter of a square of depth $i$ is bounded by $3^{k - i +1} + 3\cdot 3^{k-i} = 2\cdot 3^{k-i+1}$. 
		
	Thus the central spanning square has area $\leq 17f(12\cdot 4\cdot 3^k)$ and spanning squares of depth $i$ have area $\leq 17f(12\cdot 2\cdot 3^{k-i+1})$.
	
	Also by \cref{lem:bigons} , we see that each bigon has perimeter $\leq 5$ and as such there is a uniform bound $D$ on the area of all bigons appearing in our proof. 
	
	We deduce that the area enclosed by $\gamma$ is
	$$\leq 17f(12\cdot 4\cdot 3^k) + \sum_{i=1}^k 4\cdot 3^{i-1}\cdot17f(12\cdot 2\cdot 3^{k-i+1}) + 4\cdot 3^k\cdot D.$$ 

	Define $f'\colon \NN\to \RR$ by $n \cdot f'(n) := f(n)/n$ and assume that it is superadditive.
	We deduce the following estimates:
	\begin{align*}
	\sum_{i=1}^k 4\cdot 3^{i-1}\cdot17f(12\cdot 2\cdot 3^{k-i+1})
	& = \sum_{i=1}^k 4\cdot 3^{i-1}\cdot17\cdot 24\cdot 3^{k-i+1}f'(24\cdot 3^{k-i+1})\\
	&\leq 4\cdot 17 \cdot 24\cdot 3^k \sum_{i=1}^k f'(24\cdot 3^{k-i+1})\\
	&\leq 4\cdot 17 \cdot 24 \cdot 3^{k}  f'(24 \sum_{i=1}^k \cdot 3^{k-i+1})\\
	&\leq 4\cdot 17 \cdot 24 \cdot 3^{k}  f'(24\cdot 3^{k+1})\\
	&\leq 4\cdot 17 f(72\cdot 3^{k})\\
	\end{align*}
	Thus we see that the area of $\gamma$ is bounded above by 
	\begin{align*}
	17f(12\cdot 4\cdot 3^k) + 4\cdot 17 f(72\cdot 3^{k}) + 4\cdot 3^k\cdot D
	&\leq 85\cdot f(72 \cdot 3^k) + 4\cdot 3^k\cdot D\\
	&\leq 85 f(72 \cdot 3 n) + 4\cdot 3n\cdot D\\
	&\preccurlyeq f(n). 
	\end{align*}
	We deduce the desired upper bound. The lower bound can be obtained by retractions from $K$ to the $G_i\times G_j$. 	

The bounds on the Dehn function when $f(n)/n$ is not super-additive can be deduced from the above in analogy to the proof of \Cref{thm:algebraictriangle}.
\end{proof}
	
\begin{proof}[Proof of Corollary \ref{cor:square-thm-3}]
	The proof is analogous to the proof of Corollary \ref{cor:triangle-thm-3}.
\end{proof}

In analogy to the strategy described at the beginning of Section \ref{sec:trianglemethod}, one could also pursue a geometric approach using Lipschitz height maps to prove a geometric analogue of Theorem \ref{thm:algebraicsquare}. As before our algebraic approach allowed us to avoid some of the subtleties that one would face in a geometric approach.

\section{Applications}
\label{sec:applications}

In this section we will explore various applications of our main results. In particular, we will prove Theorems \ref{thm:Dehn-fct-Disons-groups} and \ref{thm:Dehn-fin-props}.

\subsection{Subdirect products of free groups with quadratic Dehn function}
\label{sec:SPFs-with-quadratic-Dehn}

As we discussed in the introduction, the class of SPFs provides a natural generalisation of the Stallings--Bieri groups. Considering that the Stallings--Bieri groups have quadratic Dehn functions, one may wonder if the same applies to SPFs. In general this turns out to be far from true; indeed, there are SPFs satisfying arbitrarily large polynomial lower bounds on their Dehn functions \cite{LloTes-18-II}. However, we will now explain that the result does remain true for SPFs with sufficiently high regularity properties.

We fix $r\geq 2$ and consider a finitely presented subgroup $G\leq F_{n_1}\times \dots \times F_{n_r}$ of a direct product of non-abelian free groups. We may assume that $G$ is \emph{full} ($G\cap F_{n_i}\neq 1$ for all $i$) and \emph{subdirect} (the projection of $G$ to every factor is surjective). Bridson, Howie, Miller and Short proved that if $G$ is of type $\mathcal{F}_{r}$ then $G$ is virtually a direct product of $\leq r$ free groups \cite{BriHowMilSho-02}. In particular, this means that all interesting examples of SPFs in a direct product of $r$ free groups will not be of finiteness type $\mathcal{F}_r$. 

Finiteness properties play an important role even among subgroups that are not of type $\mathcal{F}_r$, the general idea being that the stronger the finiteness properties the more regular the group. This is for instance illustrated by the fact that every subgroup of type $\mathcal{F}_k$ with $k> \frac{r}{2}$ is {\em virtually coabelian}, meaning that there are finite index subgroups $F_{m_i}\leq F_{n_i}$, $l\in \NN$ and a surjective homomorphism $\phi: F_{m_1}\times \dots \times F_{m_r}\to \ZZ^l$ such that $\ker(\phi) = G\cap\left(F_{m_1}\times \dots \times F_{m_r}\right)\leq G$ is a finite index subgroup \cite[Corollary 3.5]{Kuc-14}. Moreover, it is not hard to see that the {\em coabelian corank} $l$ of $G$ is an invariant, i.e. does not depend on the choice of finite index subgroups and surjective homomorphism. 

One may further argue that for fixed finiteness properties regularity decreases with increasing corank, the idea being that the larger the corank, the further the group is from being a direct product. A concrete manifestation of this intuition is provided by the following application of our work. 

\begin{theorem}
\label{thm:Fmcoabelian}
 For $r\geq 2$, $m<\frac{r}{2}$, and $n_i\geq 2$, $1\leq i \leq r$, let $K\leq F_{n_1}\times \dots \times F_{n_r}$ be a full subdirect product of type $\mathcal{F}_{r-m}$. Then $G$ is virtually coabelian of corank $l\geq 0$ and if $\left\lceil \frac{l}{2}\right\rceil\leq \frac{r}{4m}$, then $\delta_G(n)\asymp n^2$.
\end{theorem}

We recall that the bound on the corank in Theorem \ref{thm:Fmcoabelian} is optimal:  Dison proved that the kernel of the canonical homomorphism $F_2\times F_2\times F_2\to \ZZ^2$ induced by the abelianization on factors satisfies a cubical lower bound on its Dehn function \cite{Dis-09}, while Theorem \ref{thm:Fmcoabelian} shows that for $r\geq 4$ the kernel of the canonical homomorphism $F_2^{\times r}\to \ZZ^2$ induced by the abelianization on factors has quadratic Dehn function.

\begin{proof}[Proof of Theorem \ref{thm:Fmcoabelian}]
 Since $r-m > \frac{r}{2}$, \cite[Corollary 3.5]{Kuc-14} implies that there is a surjective homomorphism $\phi: F_{n_1}\times \dots \times F_{n_r}\to \ZZ^l$ such that $K=\ker(\phi)$ (after possibly passing to finite index subgroups of $K$ and of the $F_{n_i}$). Since $\lceil \frac{l}{2}\rceil \leq \frac{r}{4m}$, there is a partition $\left\{1,\dots,r\right\}=\bigsqcup _{j=1}^{4\lceil \frac{l}{2}\rceil} I_j$ into $4\lceil \frac{l}{2}\rceil$ sets of size $|I_j|\geq m$. Denoting $H_{j}:= \prod_{i\in I_j} F_{n_i}$ and using that $K$ is full subdirect of type $\mathcal{F}_{r-m}$ we obtain that $\phi(H_{j})=:A_{j} \leq \ZZ^l$ is a finite index subgroup for $1\leq j \leq 4\lceil \frac{l}{2}\rceil$ (see \cite[Corollary 5.4]{Llo-17}). Thus, for $A:=\cap_{1\leq j\leq 4\lceil \frac{l}{2}\rceil } A_i$ and $H'_j := \phi^{-1}(A)\cap H_j$, we obtain that the restriction
 \[
  \phi':H'_1\times \dots \times H'_{4\lceil \frac{l}{2}\rceil}\to A=\ZZ^l
 \]
 of $\phi$ is surjective on factors (i.e. $\phi'(H'_j)=A$ for all $j$). By construction $\ker(\phi')\leq K$ is a finite index subgroup.
 
 For $1\leq i\leq 4$ define $G_i:= H'_{1+(i-1)\lceil\frac{l}{2}\rceil}\times \dots \times H'_{i\lceil\frac{l}{2}\rceil}$. Since $\phi'(H'_i)=A=\ZZ^l$, we deduce that for every factoring $P=\left\{B_1,B_2\right\}$ of $\ZZ^l$ with ${\rm rk}_{\ZZ}(B_1),{\rm rk}_{\ZZ}(B_2)\leq \lceil \frac{l}{2}\rceil$ the restriction $\phi'|_{G_i} : G_i \to A$ admits a $P$-splitting. In particular, the homomorphism 
 \[
  \phi':G_1\times \dots \times G_4\to A
 \]
 satisfies the assumptions of Theorem \ref{thm:algebraicsquare} with $f(n)=n^2$. We deduce that $\ker(\phi')$ has Dehn function $n^2$. Thus, the same holds for its finite extension $K$, completing the proof.
\end{proof}

We observe that Dison's groups $K_m^r(l)$ arise as special case of Theorem \ref{thm:Fmcoabelian}. In particular, they have Dehn function $n^2$ proving Theorem \ref{thm:Dehn-fct-Disons-groups}. 

\begin{remark}\label{rem:Fm-Dehn}
 Dison's quintic upper bound on the Dehn function of $K_m^r(l)$ in \cite[Proposition 13.3(3)]{Dis-08} holds for $l\leq \frac{r}{2}>1$. This means that Theorem \ref{thm:Dehn-fct-Disons-groups} provides the precise Dehn functions for all groups covered by  \cite[Proposition 13.3(3)]{Dis-08} with the exception of a finite number of cases for every fixed value of $l$. We expect that the Dehn functions in these cases are also quadratic and that we merely needed to exclude them for technical reasons. In fact it seems reasonable to believe that there is a variation of our techniques that provides quadratic bounds also for these cases.
\end{remark}

Arguing similarly as in the proof of Theorem \ref{thm:Fmcoabelian} we also obtain the following example of groups with quadratic Dehn function and interesting finiteness properties.
\begin{example}
 Let $l\geq 1$ and $r\geq 3l$. Let $\left\{v_1,\dots,v_r\right\}\subset \ZZ^l$ be integer valued vectors such that for $1\leq i_1<\dots < i_l\leq r$ the subset $\left\{v_{i_1},\dots,v_{i_l}\right\}\subset \ZZ^l$ is linearly independent. For surjective homomorphisms $\phi_i:F_2\to \ZZ$, $1\leq i \leq r$, define a homomorphism $\phi:= \sum_{i=1}^r v_i\cdot \phi_i : (F_2)^{\times r}\to \ZZ^l$. After passing to a finite index subgroup of $\ZZ^l$ we may assume that $\phi$ is surjective. By our assumptions the restriction of $\phi$ to any $l$ factors has image a finite index subgroup of $\ZZ^l$. We can now argue similar to the proof of Theorem \ref{thm:Fmcoabelian} that, by Theorem \ref{thm:algebraictriangle}, $\delta_K(n)\asymp n^2$. Arguing via virtual surjections to $r-l$-tuples, we can moreover show that $K$ is of type $\mathcal{F}_{r-l}$, but not of type $\mathcal{F}_{r-l+1}$ \cite[Theorem C]{Koc-10}.
\end{example}

Finally, we can provide a precise version of Remark \ref{rem:Disons-Thm-11-3-4}:
\begin{theorem}\label{thm:Disons-Thm-11-3-4}
Let $r\geq3$, let $G_1,\dots, G_r$ be finitely presented groups and let $K:=\ker(G_1\times \dots \times G_r\stackrel{\phi}{\rightarrow} \ZZ^l)$ be a coabelian subgroup of corank $l\geq 0$. Denote by $f$ the Dehn function of $G_1\times \dots \times G_r$. If $\left\lceil \frac{l}{2}\right\rceil\leq \frac{r}{4}$ and the restriction of $\phi$ to every factor is virtually surjective, then $f(n)\preccurlyeq \delta_K(n)\preccurlyeq log (n) \cdot \overline{f} (n)$. If, moreover, $f(n)/n$ is superadditive, then $\delta_K(n)\asymp f(n)$.
\end{theorem}
\begin{proof}
The proof is very similar to the proof of Theorem \ref{thm:Fmcoabelian} for $m=1$. Indeed, by our assumptions $\phi|_{G_i}:G_i\to \ZZ^l$ for $1\leq i \leq r$ is virtually surjective and we then argue as in the proof of Theorem \ref{thm:Fmcoabelian} that we can apply Theorem \ref{thm:algebraicsquare} to obtain the desired conclusions.
\end{proof}

To put our result into context: Dison proved that the Dehn function of a group $K$ satisfying the conditions of Theorem \ref{thm:Disons-Thm-11-3-4} satisfies an upper bound of $n\cdot \beta_1(n^2)+\beta_2(n)$ on its Dehn function, where $\beta_1(n)$ is the Dehn function of $G_1\times \dots \times G_{r-l}$ and $\beta_2(n)$ is the Dehn function of $G_{r-l+1}\times \dots \times G_r$ \cite[Theorem 11.3 (4)]{Dis-08}. Thus, our result provides a significant improvement on these bounds. The precise requirement in Dison's result is that $l\leq \frac{r}{2}>1$. This means that in analogy to Remark \ref{rem:Fm-Dehn} we need to exclude a finite number of cases for every value of $l$ when comparing to his work. Again we believe that this is merely for technical reasons and that the result should hold for all cases covered by \cite[Theorem 11.3 (4)]{Dis-08}.

\subsection{Finiteness properties and Dehn functions}\label{sec:finprops-Dehn-fcts}
We will now prove Theorem \ref{thm:Dehn-fin-props}, showing the existence of 1-ended irreducible groups of type $\mathcal{F}_{n-1}$ and not $\mathcal{F}_n$ with prescribed Dehn function. 
\Dehnfinprops
\begin{proof}[Proof of Theorem \ref{thm:Dehn-fin-props}]
Let $G$ be a group with Dehn function $f(n)\geq n^2$. By \cite[Corollary]{GubSap-99}, the Dehn function of the free product $G\ast F_2$ is the superadditive closure $\overline{f}$ of $f$. Thus, the direct product $(G\ast F_2)^{\times n}=G\ast F_2\times \dots \times G\ast F_2$ of $n\geq 3$ copies of $G\ast F_2$ has Dehn function $\overline{f}$. 

Let $\phi$ be the composition of the projection $G\ast F_2\to F_2$ with a surjective homomorphism $F_2\to \ZZ$. Let $\psi : (G\ast F_2)^{\times n}\to \ZZ$ be the unique homomorphism that restricts to $\phi$ on every factor and let $K:=\ker(\psi)$. The statements about $\delta_K(n)$ are immediate consequences of Theorem \ref{thm:algebraictriangle}.

The homomorphism $\psi$ factors through a homomorphism $\nu: (F_2)^{\times n}\to \ZZ$ which is surjective on every factor and thus has kernel $\ker(\nu)=\mathrm{SB}_n$ a Stallings--Bieri group. Because the canonical projection $(G\ast F_2)^{\times n}\to (F_2)^{\times n}$ is a retraction, the same is true for  its restriction to the surjective homomorphism $\ker(\psi)\to \ker(\phi)=\mathrm{SB}_n$. Since $\mathrm{SB}_n$ is not of type $\mathcal{F}_n$ we deduce that the same holds for $\ker(\psi)$. On the other hand the abelian case of the $n-(n+1)-(n+2)$ Conjecture \cite[Theorem 6.3]{Kuc-14} implies that $\ker(\psi)$ is of type $\mathcal{F}_{n-1}$.

The irreducibility of $K$ is a straight-forward consequence of the irreducibility of $\mathrm{SB}_n$.

Finally, it is not hard to prove that $K$ is 1-ended. Indeed, we can use the natural projections $q_i:K\to \ZZ$, $1\leq i\leq n$, induced by the maps $\phi$ on factors, to prove the existence of a path between any two points lying in (a priori distinct) unbounded components of the complement of a compact set in the Cayley graph for $K$. For this we observe that any compact subset in the Cayley graph of $K$ maps to a compact subset in the Cayley graph of $\ZZ$ under all of the $q_i$, while unbounded components map to an unbounded subset under at least one of the $q_i$. Given two points as above, we can then use two of the projections $q_i$ to construct a path between them which does not intersect the given compact set.
\end{proof}

\subsection{Applications to right-angled Artin groups}\label{sec:raags}

Given a finite graph $\Gamma$ with vertices $V(\Gamma)$ and edges $E(\Gamma)\subset V(\Gamma)\times V(\Gamma)$ we define the \emph{right-angled Artin group} (short: RAAG) $A_{\Gamma}$ by
\[
A_{\Gamma}:= \left\langle V(\Gamma)\mid \left[v,w\right]\mbox{ if } (v,w)\in E(\Gamma)\right\rangle.
\]  

The precise finiteness properties of arbitrary coabelian subgroups of right-angled Artin groups have been computed in \cite{MeiMeiVan-98,BuxGon-99}, generalising the results for Bestvina--Brady groups \cite{BesBra-97}. In contrast our understanding of their Dehn functions seems to be mostly limited to Bestvina--Brady groups \cite{Dis-08-II,CarFor-17,Cha-19,ABDDY-13}, which are the cocyclic subgroups of RAAGs obtained by mapping all generators to the same generator of $\ZZ$. It would be interesting to perform a general study of the Dehn functions of coabelian subgroups of RAAGs.

Here we provide two straight-forward applications of our results to Dehn functions of coabelian subgroups of RAAGs which admit a splitting as direct product of RAAGs. The first is obtained by applying Theorem \ref{thm:algebraicsquare} and the second by applying Theorem \ref{thm:algebraictriangle}.

\begin{theorem}\label{thm:4-factor-RAAG}
	Let $A_\Gamma = H_1\times H_2\times H_3\times H_4$ be a product of four right-angled Artin groups. 
	Suppose that we have $P$-split maps $\phi_i\colon H_i\to \ZZ^m$. 
	Let $\phi = \sum \phi_i$. 
	Then $K = \ker(\phi)$ has quadratic Dehn function. 
\end{theorem}
\begin{proof}
	This follows immediately from Theorem \ref{thm:algebraicsquare} and the fact that RAAGs have quadratic Dehn function. 
\end{proof}

\begin{remark}
	Suppose that $H_i=\ZZ^k\ast \ZZ^l$, $1\leq i \leq 4$  and that $\phi_i$ is the abelianisation map. Then the maps $\phi_i$ are $P$-split if and only if $P = \{\ZZ^k, \ZZ^l\}$. In particular, there are non-trivial applications of Theorem \ref{thm:algebraicsquare} for all choices of factoring.
\end{remark}

Using Theorem \ref{thm:algebraictriangle} instead of Theorem \ref{thm:algebraicsquare} we obtain:
\begin{theorem}\label{thm:3-factor-RAAG}
	Let $A_\Gamma = H_1\times H_2\times H_3$ be a product of three right-angled Artin groups. 
	Suppose that we have split surjections $\phi_i\colon H_i\to \ZZ^m$. 
	Let $\phi = \sum \phi_i$. 
	Then $K = \ker(\phi)$ has quadratic Dehn function. 
\end{theorem}

We also record the following simple existence condition for splittings and $P$-splittings.
\begin{proposition}\label{prop:graphical-admissibility}
	Let $\phi\colon A_\Gamma\to \ZZ^m$ be a homomorphism. 
	Let $P = \{A, B\}$
	Then $\phi$ is $P$-split if there are complete subgraphs $\Delta_1, \Delta_2$ of $\Gamma$ such that:
	\begin{itemize}
		\item $\phi\colon A_{\Delta_1\cup\Delta_2}\to \ZZ^m$ is surjective. 
		\item $\phi|_{A_{\Delta_1}}$ maps $A_{\Delta_1}$ onto $A$ and $\phi|_{A_{\Delta_2}}$ maps $A_{\Delta_2}$ onto $B$.
	\end{itemize}

	Also, $\phi$ is split, if there is a complete graph $\Delta$ on $n$ vertices such that the homomorphism $\phi\colon A_{\Delta}\to \ZZ^m$ is an isomorphism. 
\end{proposition}

Its practical use is illustrated in the following application of \cref{thm:3-factor-RAAG}. 

\begin{example}
	Let $\Lambda$ be a hexagon, i.e a triangulation of $S^1$ with 6 vertices and six edges. 
	This a bipartite graph with bipartite vertex set $V_1\cup V_2$.
	Let $\psi\colon A_\Lambda\to \ZZ^2$ be the map given by mapping generators of $V_1$ to $(1,0)$ and generators of $V_2$ to $(0,1)$.	
	By \cref{prop:graphical-admissibility} $\psi$ is a split surjection. 
	Let $\phi: A_\Gamma = A_\Lambda\times A_\Lambda\times A_\Lambda\to \ZZ^2$ be the homomorphism that restricts to $\psi$ on each factor. 
	Then $\ker(\phi)$ has quadratic Dehn function by \cref{thm:3-factor-RAAG}.
\end{example}

\bibliography{References}
\bibliographystyle{amsplain}

\end{document}